\documentclass[reqno,11pt,letterpaper]{amsart}

\usepackage[T1]{fontenc}
\usepackage[utf8]{inputenc}
\usepackage{amsmath}
\usepackage{amsthm}
\usepackage{mathtools}
\usepackage{amssymb}
\usepackage{mathrsfs}
\usepackage{bm}
\usepackage{verbatim}
\usepackage{enumerate}
\usepackage[top=1.5in,bottom=1.5in,left=0.7in,right=0.7in]{geometry}
\usepackage[colorlinks]{hyperref}
\hypersetup{
    colorlinks,
    citecolor=blue,
    linkcolor=red
}
\usepackage{xypic}
\usepackage{tikz}
\usetikzlibrary{decorations.pathreplacing}
\usepackage{xcolor}
\usepackage{float}
\usepackage{cleveref}
\usepackage{comment}
\usepackage[abs]{overpic}
\usepackage{aliascnt}

\numberwithin{equation}{section}
\newtheorem{theorem}{Theorem}[section]
\newtheorem*{theorem*}{Theorem}
\newtheorem{corollary}[theorem]{Corollary}
\newtheorem{proposition}[theorem]{Proposition}
\newtheorem{lemma}[theorem]{Lemma}
\newtheorem{question}[theorem]{Question}

\newtheorem*{claim*}{Claim}

\newtheorem{remark}[theorem]{Remark}

\newtheorem{example}[theorem]{Example}
\newtheorem{conjecture}[theorem]{Conjecture}
\newtheorem{definition}[theorem]{Definition}

\newcommand{\R}{\mathbb R}
\newcommand{\Z}{\mathbb Z}
\newcommand{\Q}{\mathbb Q}
\newcommand{\C}{\mathbb C}
\newcommand{\D}{\mathbb D}
\newcommand{\B}{\mathbb B}
\newcommand{\bbS}{\mathbb S}
\newcommand{\CP}{\mathbb C\mathbb P}
\newcommand{\Algebra}{\mathfrak{A}}

\newcommand{\cA}{\mathcal A}

\newcommand{\cM}{\mathcal M}

\newcommand{\rd}{\mathrm d}
\newcommand{\std}{\mathrm{std}}

\DeclareMathOperator{\Ima}{im}
\DeclareMathOperator{\virdim}{virdim}
\DeclareMathOperator{\CZ}{CZ}
\DeclareMathOperator{\SFT}{SFT}
\DeclareMathOperator{\CO}{CO}
\DeclareMathOperator{\CH}{CH}
\DeclareMathOperator{\CC}{CC}
\DeclareMathOperator{\LCH}{LCH}
\DeclareMathOperator{\tb}{tb}
\DeclareMathOperator{\sign}{sign}

\newcommand{\Addresses}{{%
\bigskip
\footnotesize

Russell Avdek,\par\nopagebreak
\textsc{Sorbonne University, Paris Cité University, CNRS, IMJ-PRG, F-75005 Paris, France}\par\nopagebreak
\textit{E-mail address}: \href{mailto:russell.avdek@imj-prg.fr}{russell.avdek@imj-prg.fr}

\medskip

Youlin Li,\par\nopagebreak
\textsc{School of Mathematical Sciences, Shanghai Jiao Tong University, Shanghai, China}\par\nopagebreak
\textit{E-mail address}: \href{mailto:liyoulin@sjtu.edu.cn}{liyoulin@sjtu.edu.cn}

\medskip

Zhengyi Zhou,\par\nopagebreak
\textsc{State Key Laboratory of Mathematical Sciences, Chinese Academy of Sciences;}\par\nopagebreak
\textsc{Morningside Center of Mathematics, Chinese Academy of Sciences;}\par\nopagebreak
\textsc{Academy of Mathematics and Systems Science, Chinese Academy of Sciences, China}\par\nopagebreak
\textit{E-mail address}: \href{mailto:zhyzhou@amss.ac.cn}{zhyzhou@amss.ac.cn}

}}

\title{Algebraic overtwistedness of contact $(+1)$-surgeries via augmented Floer theories}

\author{Russell Avdek}
\author{Youlin Li}
\author{Zhengyi Zhou}

\begin{document}

\maketitle

\begin{abstract}
We prove that contact $(+1)$-surgery along a Legendrian $n$-sphere
\(\Lambda\subset(\bbS^{2n+1},\xi_{\std})\) with
\[
\tb(\Lambda)\ne (-1)^{\frac{n(n-1)}{2}+1}
\]
has vanishing contact homology. In dimension three, this produces new examples of tight algebraically overtwisted contact manifolds, vastly extending the authors' previous results: If a Legendrian knot in \((\bbS^3,\xi_{\std})\) admits an exact, oriented Lagrangian filling that is not a disk, then contact $(+1/k)$-surgery on it is algebraically overtwisted and tight for every positive integer \(k\). Our proofs use versions of symplectic cohomology and wrapped Floer cohomology defined by augmentations of a chain-level contact homology algebra.
\end{abstract}

\setcounter{tocdepth}{1}
\tableofcontents

\section{Introduction}

Contact homology is a powerful invariant of contact manifolds defined within the framework of symplectic field theory. It was first introduced by Eliashberg, Givental, and Hofer \cite{EliGivHof}, and given rigorous foundations by Pardon \cite{Pardon:Contact} and Bao--Honda \cite{Bao-Honda}. This invariant has proved useful in distinguishing contact structures. For example, Yau \cite{Yau} and Bourgeois--van Koert \cite{BvK} showed that the contact homology of an overtwisted contact manifold vanishes; see also \cite{Z+1}. Consequently, non-vanishing contact homology provides an effective criterion for tightness in
applications \cite{BGMZ,AZ}.

A fundamental method for constructing contact manifolds is contact surgery. Ding--Geiges \cite{DingGeiges:consur} proved that every closed contact $3$-manifold can be obtained from the standard contact sphere \((\bbS^3,\xi_{\std})\) by contact $(\pm 1)$-surgeries on a Legendrian link. This result was later extended by Conway--Etnyre \cite{CE20}, who showed that any closed contact manifold (of any odd dimension) can be constructed as the result of surgeries along isotropic and coisotropic spheres in the standard contact sphere.

The interaction between contact surgery and contact homology reveals several interesting phenomena. \emph{Legendrian surgery}, that is, contact $(-1)$-surgery along a Legendrian sphere in a standard contact sphere, always produces a Stein fillable contact manifold and hence one with non-vanishing contact homology. On the other hand, contact $(+1)$-surgery along a Legendrian sphere may produce a contact manifold with vanishing contact homology. For instance, work of Ozbagci \cite{Ozbagci} and of Casals--Murphy--Presas \cite{CMP} shows that performing contact $(+1)$-surgery on a stabilized Legendrian knot in dimension \(3\), or on a loose Legendrian in dimensions at least \(5\), always yields an overtwisted contact manifold. Therefore the result of such a $(+1)$-surgery must have vanishing contact homology and cannot admit a symplectic filling \cite{Gromov:JCurves, Niederkruger:Plastik}.

Motivated by these vanishing results for overtwisted manifolds, Bourgeois--Niederkr\"uger \cite{BN} defined contact manifolds with vanishing contact homology to be \emph{algebraically overtwisted}. It is natural to ask whether every algebraically overtwisted contact manifold is overtwisted. The first examples of closed contact manifolds that are tight but algebraically overtwisted were found by the first author in \cite{avdek2020combinatorial}; these examples arise from contact $(+1/k)$-surgery on the Legendrian right-handed trefoil in \((\bbS^3,\xi_{\std})\), for positive integers \(k\). This family was subsequently generalized by the second and third authors to surgeries on Legendrian positive torus knots \cite{Li-Zhou}.

In this paper we study the contact homology of contact manifolds obtained by contact $(+1)$-surgery on Legendrian spheres in standard contact spheres of any odd dimension. We take a different approach from \cite{avdek2020combinatorial}. Let \(\Lambda\) be a Legendrian \(n\)-sphere in \((\bbS^{2n+1},\xi_{\std})\), and denote by \((\bbS^{2n+1}_{+1}(\Lambda),\xi_{+1}(\Lambda))\) the result of contact $(+1)$-surgery along \(\Lambda\). After identifying the complement of a point in \((\bbS^{2n+1}, \xi_{\std})\) with \((\R^{2n+1},\xi_{\std})\), the Thurston--Bennequin invariant \(\tb(\Lambda)\) can be defined as the signed count of self-intersections in the Lagrangian projection of \(\Lambda\) \cite[Appendix A]{Murphy:LooseLegendrian}.

Our main theorem shows that \((\bbS^{2n+1}_{+1}(\Lambda),\xi_{+1}(\Lambda))\) is algebraically overtwisted for a broad class of Legendrian spheres.

\begin{theorem}\label{thm:vanishingCH}
Let \(\Lambda\) be a Legendrian sphere in
\((\bbS^{2n+1},\xi_{\std})\), where \(n\ge 1\). If
\[
\tb(\Lambda)\ne (-1)^{\frac{n(n-1)}{2}+1},
\]
then \((\bbS^{2n+1}_{+1}(\Lambda),\xi_{+1}(\Lambda))\) is algebraically overtwisted.
\end{theorem}

\subsection{Corollaries of \Cref{thm:vanishingCH}}

As an immediate consequence of \Cref{thm:vanishingCH}, we obtain the following obstruction to symplectic fillability.

\begin{corollary}\label{cor:nofilling1}
If \(\Lambda\subset(\bbS^{2n+1},\xi_{\std})\) is a Legendrian sphere with
\[
\tb(\Lambda)\ne (-1)^{\frac{n(n-1)}{2}+1},
\]
then
\((\bbS^{2n+1}_{+1}(\Lambda),\xi_{+1}(\Lambda))\) admits no weak fillings.
\end{corollary}

\begin{proof}
The results are only relevant when $n$ is odd, in which case it is easily checked that $H^{2}(\bbS^{2n+1}_{+1}(\Lambda), \R)=0$. Hence any weak filling of $(\bbS^{2n+1}_{+1}(\Lambda),\xi_{+1}(\Lambda))$ may be deformed to a strong filling \cite[Proposition 6]{MNW12}. According to \cite[Corollary 7]{MNW12} (which generalizes  the $\dim = 3$ case of \cite[Corollary 6]{NW:WeakFillings}), the vanishing of $CH(\bbS^{2n+1}_{+1}(\Lambda),\xi_{+1}(\Lambda), \C)$ (provided by \Cref{thm:vanishingCH}) obstructs the existence of strong symplectic fillings.
\end{proof}

For contact \(3\)-manifolds, \Cref{cor:nofilling1} also follows from the main result of \cite{CET}. In this dimension, we obtain a new family of tight but algebraically overtwisted contact manifolds, substantially extending the examples of \cite{avdek2020combinatorial} and \cite{Li-Zhou}. The tightness of the examples described follows from \cite{Golla,LS}, which guarantee the non-vanishing of contact Ozsv\'ath--Szab\'o invariants.

\begin{corollary}\label{cor:3D}
Let \(\Lambda\) be a Legendrian knot in \((\bbS^{3},\xi_{\std})\). Suppose that
\[
\tb(\Lambda)\neq -1
\qquad\text{and}\qquad
\tb(\Lambda)=2g_s(\Lambda)-1,
\]
where \(g_s(\Lambda)\) denotes the smooth slice genus of the underlying
topological knot type. Then, for every positive integer \(k\), the contact
manifold
\[
(\bbS^{3}_{+1/k}(\Lambda),\xi_{+1/k}(\Lambda))
\]
is algebraically overtwisted and tight. In particular, if \(\Lambda\) bounds an
exact, oriented Lagrangian surface that is not a disk, then
\((\bbS^{3}_{+1/k}(\Lambda),\xi_{+1/k}(\Lambda))\) is algebraically overtwisted
and tight for every positive integer \(k\).
\end{corollary}

Bourgeois--Sabloff--Traynor \cite[Theorem 1.3]{BourSabTra} constructed
Lagrangian-fillable Legendrian spheres in \((\bbS^{4n-1},\xi_{\std})\) with
arbitrary odd Thurston--Bennequin invariants. \Cref{thm:vanishingCH}
applies to these Legendrians whenever their Thurston--Bennequin invariants are
not equal to \((-1)^n\). In light of \Cref{cor:3D}, it is natural to
ask the following higher-dimensional question.

\begin{question}
For \(n\geq 2\), does contact $(+1)$-surgery along a Lagrangian-fillable Legendrian sphere
in \((\bbS^{2n+1},\xi_{\std})\) always produce a tight contact manifold?
\end{question}

While \cite{BourSabTra} addresses Legendrian spheres in \((\bbS^{4n-1},\xi_{\std})\), it follows from
\cite[Proposition 3.2]{EES} that for any Legendrian sphere
\(\Lambda\subset(\bbS^{4n+1},\xi_{\std})\), \(n\geq 1\),
\[
\tb(\Lambda)=(-1)^{n+1}.
\]
Thus \Cref{thm:vanishingCH} does not apply to Legendrian spheres in \((\bbS^{4n+1},\xi_{\std})\), \(n\geq 1\). Nevertheless, the strategy of its proof can be adapted to yield consequences closely related to the vanishing of contact homology.

\begin{proposition}\label{thm:vanishingCH1}
Let \((\bbS^{2n+1},\Lambda,\xi_{\std})\), \(n\ge 2\), be an asymptotically
dynamically convex pair in the sense of \cite[Definition 5.2]{Lazarev:flexiblefillings}.
Suppose that \(\Lambda\) admits an exact spin Lagrangian filling
\(L\subset(\D^{2n+2},\omega_{\std})\) such that the Maslov class of \(L\)
vanishes and
\[
H^*(L;\C)\ne \C.
\]
Then, for any choice of auxiliary data, the contact DGA of
\((\bbS^{2n+1}_{+1}(\Lambda),\xi_{+1}(\Lambda))\) admits no
\(\Z\)-graded augmentation over \(\C\).
\end{proposition}

A rich source of such asymptotically dynamically convex pairs comes from the \(h\)-principles of Eliashberg--Ganatra--Lazarev \cite[Corollary 4.5]{EGL:flexiblelagrangians}. In particular, by \cite[Theorem 4.7]{EGL:flexiblelagrangians}, when \(n=2\), the filling \(L\) may be obtained from any closed oriented \(3\)-manifold that is not a rational homology sphere by removing a ball. Moreover, the same theorem asserts that one may assume \(\partial L\) is formally isotopic to the standard Legendrian unknot. Inspired by these results, we pose the following conjecture.

\begin{conjecture}
Let \(n\ge 2\). If a Legendrian sphere
\(\Lambda\subset(\bbS^{2n+1},\xi_{\std})\) admits an exact Lagrangian filling
\(L\) that is not diffeomorphic to a ball, then
\((\bbS^{2n+1}_{+1}(\Lambda),\xi_{+1}(\Lambda))\) is algebraically overtwisted.
\end{conjecture}

\subsection{Ingredients in our proofs}

Our proof strategies are quite different from those of \cite{avdek2020combinatorial, Li-Zhou}. The proofs of the algebraic overtwistedness of the contact \(3\)-manifolds in those articles rely largely on combinatorial methods whose applicability rests on the positivity of intersections of holomorphic curves in symplectic $4$-manifolds.

The main ingredients in the proofs of this article are versions of symplectic cohomology and wrapped Floer cohomology (reviewed in \S \ref{s3}) defined using augmentations of chain-level contact homology algebras in \S \ref{s4}. Typically such invariants are defined for a Liouville filling of a contact manifold (for symplectic cohomology) and exact Lagrangians with Legendrian boundary inside of such fillings (for wrapped Floer cohomology). We will apply such cohomology theories to surgery cobordisms $W$ from $(\bbS^{2n+1}_{+1}(\Lambda),\xi_{+1}(\Lambda))$ to $(\bbS^{2n+1},\xi_{\std})$, using special contact forms described in \S \ref{s2}.

To establish \Cref{thm:vanishingCH}, we argue by contradiction. Here is a high-level sketch of the argument: If the chain-level contact homology algebra of the surgered manifold \((\bbS^{2n+1}_{+1}(\Lambda),\xi_{+1}(\Lambda))\) admitted an augmentation (with \( \C\) coefficients and \( \Z/2\Z\)-valued grading), then we would be able to apply our augmented wrapped Floer cohomology to prove that \(\tb(\Lambda) = (-1)^{\frac{n(n-1)}{2}+1}\), yielding a contradiction. Hence the algebra in question does not admit an augmentation. On the other hand, this algebra admits an augmentation if and only if its homology is non-zero. This last statement is a consequence of \Cref{Lemma:aug}, which essentially follows from the weak Nullstellensatz.

\subsection*{Acknowledgments}
We thank Georgios Dimitroglou Rizell for suggesting \Cref{Lemma:aug} and Austin Christian for interesting discussions.

R.A. is supported by the Centre National de la Recherche Scientifique (CNRS). Y.L. is supported by the National Natural Science Foundation of China under Grant No.\ 12271349. Z.Z. is supported by the National Key R\&D Program of China under Grant No.\ 2023YFA1010500, and by the National Natural Science Foundation
of China under Grant Nos.\ 12625108, 12288201 and 12231010.
\section{Topological and dynamical preliminaries}\label{s2}

Here we provide a brief overview of Legendrians in $(\bbS^{2n+1},\xi_{\std})$, contact $(+1)$-surgery, and ADC Legendrians. \Cref{prop:contact} constructs a contact form on $(\bbS^{2n+1},\xi_{\std})$ which will be used throughout the article.

\subsection{Legendrian spheres in $(\R^{2n+1},\xi_{\std})$ and
$(\bbS^{2n+1},\xi_{\std})$}\label{SS2.1}

Let \(\B^{2n}(R)\) denote the radius-\(R\) disk in
\(\C^n=\R^{2n}\) centered at the origin. We consider a Legendrian sphere
\(\Lambda\subset(\R^{2n+1},\xi_{\std})\), where
\[
\xi_{\std}=\ker \alpha_{\std}, \quad \alpha_{\std}=\rd z-\sum_{i=1}^n y_i\,\rd x_i.
\]
Since the map
\[
(z,x_i,y_i)\mapsto (r^2z,rx_i,ry_i),\qquad r>0,
\]
is a contact dilation, we may assume, up to Legendrian isotopy, that
\[
\Lambda\subset (0,1)_z\times \B^{2n}(1)\subset \R_z\times \R^{2n}.
\]
Consequently, all Reeb chords \(c\) contained in this Darboux chart satisfy
\(\int c^*\alpha_{\std}<1\). We assume throughout that all such chords are
nondegenerate.

Choose a Liouville form \(\lambda\) on \(\C\) satisfying the following conditions:
\begin{enumerate}
\item \(\lambda=-y\,\rd x\) on \(\B^2(1)\).
\item \(\lambda=\frac{1}{2}(x\,\rd y-y\,\rd x)\) outside \(\B^2(2)\).
\end{enumerate}
We may view \(\Lambda\) as lying in
\[
\left((0,1)_z\times \B^{2n}(2),\,
\rd z+\sum_{i=1}^n\pi_i^*\lambda\right).
\]
Here \(\pi_i:\C^n\to \C\) is the projection to the \(i\)-th factor. We compactify this to a strict contact manifold
\[
\begin{gathered}
(\bbS^{2n+1},\alpha)
=
\partial\left(\B^2(3)\times \B^{2n-2}(3)\times \B^2(2),
\lambda_{\mathrm{prod}}\right),\\
\lambda_{\mathrm{prod}}
=
\frac{1}{2}\rho^2\,\rd\theta
+
\sum_{i=1}^{n-1}\pi_i^*\lambda+\pi_n^*\lambda.
\end{gathered}
\]
Here \((\rho,\theta)\) are polar coordinates on \(\B^2(3)\). One must, of course, smooth the corners which is discussed in \cite[\S 2.1]{product}. The Legendrian \(\Lambda\) embeds into this compactification by
\[
(0,1)_z\times \B^{2n}(1)
\longrightarrow
\bbS^1_\theta\times \B^{2n-2}(3)\times \B^2(2),
\qquad
(z,x)\longmapsto \left(\frac{2}{9}z,x\right).
\]
Under this compactification, all Reeb chords of \(\Lambda\) (with respect to $\alpha$) have action either
\begin{enumerate}
\item less than \(1\), in which case they are contained in the original Darboux chart, or 
\item greater than \(9\pi-1\).
\end{enumerate}

Following \cite[\S 2.1]{product}, after smoothing corners the contact boundary of
\[
\left(\B^2(3)\times \B^{2n-2}(3)\times \B^2(2),
\lambda_{\mathrm{prod}}\right)
\]
has the following properties:
\begin{enumerate}
\item All Reeb orbits have period at least \(4\pi\).
\item There is a unique Reeb orbit \(\gamma_{\min}\) of period \(4\pi\), contained in
\[
\B^2(3)\times \B^{2n-2}(3)\times \partial\B^2(2).
\]
Its Conley--Zehnder index is \(n+2\), and it kills the unit in the symplectic cohomology of \(\B^2(3)\times \B^{2n-2}(3)\times \B^2(2)\) by \cite[\S 2.4, 2.6]{product}.
\end{enumerate}

Let \(\Lambda'\) be a push-off of \(\Lambda\) along the Reeb vector field, that is, the \(z\)-direction of the Darboux chart which contains $\Lambda$. The Thurston--Bennequin invariant \(\tb(\Lambda)\) is the linking number of \(\Lambda\) and \(\Lambda'\). Equivalently, it is the signed count of self-intersections in the Lagrangian projection
\[
\pi:\R^{2n+1}\to \R^{2n},\qquad
(z,x_1,y_1,\ldots,x_n,y_n)\mapsto (x_1,y_1,\ldots,x_n,y_n).
\]
Each self-intersection corresponds to a Reeb chord \(c\) contained in our Darboux chart, whose endpoints, \(a\) and \(b\), satisfy \(z(a)>z(b)\). Let
\[
V_a=\rd\pi(T_a\Lambda),\qquad V_b=\rd\pi(T_b\Lambda).
\]
After fixing an orientation of \(\Lambda\), we define \(\sign(c)=1\) if the
orientation of \(V_a\oplus V_b\) agrees with the orientation of \(\C^n\), and
\(\sign(c)=-1\) otherwise. Then
\[
\tb(\Lambda)=\sum_c \sign(c).
\]

By \cite[Proposition 3.2]{EES}, when \(n\) is even, one has
\[
\tb(\Lambda)=(-1)^{n/2+1}.
\]
Moreover, \cite[Proposition 3.3]{EES} gives
\[
\tb(\Lambda)
=
(-1)^{\frac{(n-2)(n-1)}{2}}
\sum_c (-1)^{|c|_{\SFT}}.
\]
Here \(|c|_{\SFT}\) denotes the SFT degree of the chord \(c\) used to define Chekanov--Eliashberg algebras and Legendrian contact homology in
\cite[\S 2.3]{EES}.\footnote{In \cite[\S 2.3]{EES}, \(|c|_{\SFT}\) is called
the grading and is denoted by \(|c|\). It equals the Conley--Zehnder index of
\(c\) minus \(1\).}

We summarize the above setup as follows.

\begin{proposition}\label{prop:contact}
Let \(\Lambda\subset(\bbS^{2n+1},\xi_{\std})\) be a Legendrian sphere. Then
there exists a contact form \(\alpha\) on \(\bbS^{2n+1}\) for which the following conditions are satisfied:
\begin{enumerate}
\item The shortest Reeb orbit has period \(4\pi\) and Conley--Zehnder index
\(n+2\).
\item For any exact filling \((W,\lambda)\) of \((\bbS^{2n+1},\xi_{\std})\) with \(\lambda|_{\partial W}=\alpha\), the positive symplectic cohomology
\(SH^{*,<4\pi+\delta}_+(W;\Q)\), for \(0<\delta\ll 1\), has rank two, and the composition
\[
SH^{*,<4\pi+\delta}_+(W;\Q)
\longrightarrow H^{*+1}(W;\Q)
\longrightarrow H^0(W;\Q)
\]
is surjective.
\item Up to Legendrian isotopy, all Reeb chords of \(\Lambda\) have action
either less than \(1\) or greater than \(9\pi-1>4\pi\), and the signed count
of the Reeb chords with action less than \(1\) equals \(\tb(\Lambda)\).
\end{enumerate}
\end{proposition}

Positive symplectic cohomology will be recalled in \S \ref{s3}. The first and second items above follow from \cite[Proposition 3.2 and proof of Theorem 1.1]{product} applied to \(V=\B^2(3)\times \B^{2n-2}(3)\). In the terminology of \cite{product}, the unit in symplectic cohomology is killed by the check generator \(\check{\gamma}_{\min}\) associated to the shortest Reeb orbit \(\gamma_{\min}\) on \(\partial W\). Moreover, \(SH^{*,<4\pi+\delta}_+(W;\Q)\) is generated by the two Morse--Bott perturbation generators \(\check{\gamma}_{\min}\) and \(\hat{\gamma}_{\min}\).

\subsection{Contact $(+1)$-surgery}

We now review contact $(+1)$-surgery following the exposition of Conway--Etnyre \cite{CE20}. The construction is implicit in the theory of Weinstein handle attachments
\cite{CE12,Eli90,Weinstein}.

Consider the standard symplectic form
\[
\omega
=
\sum_{i=1}^k \rd q_i\wedge \rd p_i
+
\sum_{i=1}^{N-k}\rd x_i\wedge \rd y_i
\]
on \(\R^{2k}\times \R^{2N-2k}\). The Liouville vector field
\[
v
=
\sum_{i=1}^k
\left(-p_i\partial_{p_i}+2q_i\partial_{q_i}\right)
+
\frac{1}{2}\sum_{i=1}^{N-k}
\left(x_i\partial_{x_i}+y_i\partial_{y_i}\right)
\]
will be used to define a model Weinstein \(k\)-handle. For \(a,b>0\), let \(\D_a\) be the
radius-\(a\) disk in the \(p_i\)-subspace, and let \(\D_b\) be the radius-\(b\)
disk in the \((q_i,x_i,y_i)\)-subspace. Then
\[
H_{a,b}=\D_a\times \D_b
\]
is the standard (square) Weinstein \(k\)-handle. The Liouville vector field, $v$, points
out of $H_{a,b}$ along
\[
\partial_+H_{a,b}:=\D_a\times \partial\D_b
\]
and into $H_{a,b}$ along
\[
\partial_-H_{a,b}:=\partial\D_a\times \D_b,
\]
thereby inducing contact structures on the smooth boundary strata of $H_{a,b}$.

The sphere
\[
\bbS_a=\partial\D_a\times \{0\}
\]
is isotropic, while
\[
\bbS_b=\{0\}\times \partial\D_b
\]
is coisotropic. By Moser's trick, the germ of the contact structure along \(\bbS_a\subset\partial_-H_{a,b}\) is contactomorphic to the germ along any isotropic sphere with trivial conformal symplectic normal bundle in a contact manifold. Similarly, the germ along \(\bbS_b\subset\partial_+H_{a,b}\) is contactomorphic to the germ along a standard
coisotropic sphere in any contact manifold \cite[Lemma 3.4]{CE20}.

Consider an isotropic sphere in a contact manifold \(Y\), with a chosen trivialization of its conformal symplectic normal bundle. We may attach \(H_{a,b}\) to the positive boundary of a finite symplectization of \(Y\) along \(\partial_-H_{a,b}\), for \(b\ll 1\), by identifying this boundary stratum of $H_{a,b}$ with a Weinstein neighborhood of the sphere. The result is a Weinstein cobordism with concave boundary \(Y\) and convex boundary \(Y'\). This operation is called \emph{isotropic surgery}.\footnote{In order to obtain a Weinstein cobordism with smooth boundary, one must address the smoothing of the corners, $\partial \partial_{-}H_{a,b}$. We ignore the details of this smoothing, which may be found in the literature.}

Conversely, given a coisotropic sphere in \(Y\), one can attach \(H_{a,b}\) to the negative boundary of a finite symplectization of $Y$ by identifying
\(\partial_+H_{a,b}\), for \(a\ll 1\), with a neighborhood of the sphere. This surgery produces a Weinstein cobordism with convex boundary \(Y\) and concave boundary \(Y''\). This operation is called \emph{coisotropic surgery}.

These operations are inverse to one another: Isotropic surgery can be undone by coisotropic surgery on the corresponding sphere in the handle, and conversely \cite[Lemma 3.9 and Proposition 3.10]{CE20}. By \cite[Theorem 1.1]{CE20}, every closed contact \((2N-1)\)-manifold can be constructed by applying a sequence of such surgeries to contact manifolds, starting with \((\bbS^{2N-1},\xi_{\std})\).

By the functoriality of contact homology with respect to Weinstein cobordisms, only coisotropic surgeries can kill contact homology. In contrast with isotropic spheres \(\bbS^{k-1}\) with \(k<N\), coisotropic spheres do not satisfy an \(h\)-principle \cite[Theorem 1.3 and Remark 6.1]{GNW}; this makes the existence of contact manifolds obtained by coisotropic surgery more subtle. Legendrian spheres, which are both isotropic and coisotropic, exhibit both flexible and rigid behavior: Loose Legendrians satisfy an \(h\)-principle \cite{Murphy:LooseLegendrian}, while non-loose Legendrians display various forms of symplectic rigidity.

\begin{definition}
\emph{Contact $(+1)$-surgery} along a Legendrian sphere \(\Lambda\subset Y\) is
coisotropic surgery along \(\Lambda\). We denote the contact manifold resulting from this surgery by \(Y_{+1}(\Lambda)\). The associated Weinstein cobordism from
\(Y_{+1}(\Lambda)\) to \(Y\) will be denoted by \(W_{\Lambda}\).
\end{definition}

\begin{remark}
Contact $(+1)$-surgery a priori depends on a parameterization \(\bbS^{N-1}\simeq \Lambda\) (up to isotopy) of the Legendrian sphere, and the resulting contact manifold, even its smooth structure, may depend on this choice. \Cref{thm:vanishingCH} is independent of such a choice of parametrization, and so we suppress it from our notation.
\end{remark}

In \cite{Avdek21}, the first author gave an alternative definition of $(+1)$-surgery using a Dehn--Seidel twist. This approach was used in \cite{CMP} to show that a contact manifold is overtwisted if and only if it can be described as the result of contact $(+1)$-surgery on a loose Legendrian.

\subsection{ADC Legendrians}

Asymptotically dynamically convex (ADC) contact manifolds and Legendrian submanifolds
were introduced by Lazarev in \cite{Lazarev:flexiblefillings}.

\begin{definition}[{\cite[Definition 3.6]{Lazarev:flexiblefillings}}]
A contact manifold \((Y^{2N-1},\xi)\) with \(c_1(\xi)=0\) is called
\emph{asymptotically dynamically convex (ADC)} if the following condition holds: There exists
\begin{enumerate}
\item a sequence of contact forms
\(\alpha_1\ge \alpha_2\ge \cdots\) for \((Y^{2N-1},\xi)\)\footnote{That is, \(\alpha_i=f_i\alpha_{i+1}\) for functions \(f_i\ge 1\)} and
\item an increasing sequence
\(D_1<D_2<\cdots\), with \(\lim_i D_i=\infty\),
\end{enumerate}
such that all
\emph{contractible} Reeb orbits of \(\alpha_i\) with period at most \(D_i\) are nondegenerate
and have positive SFT degree, i.e.
\[
\mu_{\CZ}+N-3>0.
\]
\end{definition}

\begin{remark}
Different variants of ADC exist, depending on
the vanishing assumptions for the first Chern class and on the set of Reeb
orbits considered; see \cite[\S 3]{intersection}. For the standard contact
sphere, these variants coincide, so we use Lazarev's original definition.
\end{remark}

The main result of \cite{Lazarev:flexiblefillings} states that the contact
boundary of any flexible Weinstein domain \(W\) with \(c_1(W)=0\) is ADC. The
key example for this paper is the standard contact sphere
\((\bbS^{2n+1},\xi_{\std})\), for \(n\ge 1\). Lazarev also extended the notion
of ADC to Legendrian submanifolds and pairs.

\begin{definition}[{\cite[Definitions 5.1 and 5.2]{Lazarev:flexiblefillings}}]\label{def:ADC_pair}
A connected Legendrian \(\Lambda\subset(Y,\xi)\), with \(c_1(\xi)=0\) and
\[
c_1(Y,\Lambda,\xi)=0\in H^2(Y,\Lambda;\Z),
\]
is called \emph{ADC} if the following condition holds: There exists
\begin{enumerate}
\item a sequence of contact forms
\(\alpha_1\ge \alpha_2\ge \cdots\) and
\item an increasing sequence \(D_1<D_2<\cdots\) with \(\lim_i D_i=\infty\)
\end{enumerate}
such that all Reeb chords $c$ of \(\Lambda\) with respect to \(\alpha_i\) having action \( \leq D_i\) and representing the trivial homotopy class in \(\pi_1(Y,\Lambda)\) are nondegenerate and have positive SFT degree. That is, 
\[
|c|_{\SFT}=\mu_{\CZ}-1>0.
\]
The triple \((Y,\Lambda,\xi)\) is called an \emph{ADC pair} if both the Reeb orbits and the Reeb chords satisfy the corresponding positivity conditions for the chosen sequences of contact forms and period/action thresholds.
\end{definition}

\begin{remark}
In \cite[Definition 5.1]{Lazarev:flexiblefillings}, it is also assumed that on a neighborhood \(U\) of \(\Lambda\), one has \(\alpha_i|_U=c_i\alpha_1|_U\) for constants \(c_i\). This condition is used to prove \cite[Proposition 5.5]{Lazarev:flexiblefillings}, but we will not need it in this article.
\end{remark}

\begin{example}
Let \(L\) be an \((n+1)\)-dimensional manifold with boundary \(\bbS^n\). Given
a complex trivialization
\[
\eta:TL\otimes \C \to L\times \C^{n+1},
\]
it follows from \cite[Corollary 4.5]{EGL:flexiblelagrangians} that there exists a
flexible Lagrangian embedding
\[
(L,\partial L)\subset(\D^{2n+2},\bbS^{2n+1}).
\]
Assume that \(L\) is Morse-subcritical, meaning that there is a Morse function \(f:L\to \R\) with no critical points of index \(\geq n\), such that \(\partial L\) is a regular level set of \(f\) and \(\partial_r f>0\). Here \(r\) is the outward-pointing collar coordinate near \(\partial L\). By \cite[Propositions 7.29 and 7.30, arXiv version]{Lazarev:flexiblefillings}, \((\bbS^{2n+1},\Lambda,\xi_{\std})\) is an ADC pair, where \(\Lambda=\partial L\) is the Legendrian boundary of \(L\). There are many examples of such \(L\) with \(H^*(L;\Q)\ne \Q\); see \cite[\S 4]{EGL:flexiblelagrangians}. Thus \Cref{thm:vanishingCH1} applies to
these examples.
\end{example}

\section{Symplectic cohomology, wrapped Floer cohomology, and contact homology}\label{s3}

In this section we briefly review symplectic cohomology, wrapped Floer cohomology, and contact homology. In \Cref{cor:nofilling} we use symplectic and wrapped Floer cohomologies to prove that the contact manifolds $\bbS^{2n+1}_{+1}(\Lambda)$ of \Cref{thm:vanishingCH} do not admit exact symplectic fillings. This is a warm-up for our proof of \Cref{thm:vanishingCH}, which will appear in \S \ref{s5}. In \S \ref{ss:CHreview} we review contact homology and prove the algebraic \Cref{Lemma:aug}, which is used to equate the vanishing of certain contact homology algebras with the non-existence of augmentations of the associated chain-level CDGAs.

Our conventions follow \cite{Ritter1} for symplectic cohomology, \cite{Abouzaid-Seidel,Ritter1} for wrapped Floer cohomology, and \cite{Pardon:Contact,Bao-Honda} for contact homology.

\subsection{Symplectic cohomology}\label{SS3.1}

Let \((W^{2N},\lambda)\) be a Liouville domain, and let
\[
(\widehat{W},\widehat{\lambda})
=
\bigl(W \cup ((1,\infty)_r \times \partial W) ,\,
\lambda \cup r\lambda|_{\partial W}\bigr)
\]
be its completion, obtained by attaching a half-infinite cylindrical end to the boundary. Let \(H\) be a time-dependent Hamiltonian on
\((\widehat{W},\widehat{\lambda})\). The symplectic action of a loop \(\gamma\)
is
\begin{equation}\label{eqn:action}
\cA_H(\gamma):=-\int \gamma^*\widehat{\lambda}
+\int_{S^1}H\circ \gamma\,\rd t.
\end{equation}
Our convention for the Hamiltonian vector field is
\[
\iota_{X_H}\rd\widehat{\lambda}=-\rd H.
\]
We say that an almost complex structure \(J\) is cylindrically convex near
\(\partial W\times\{r_0\}\) if, near the hypersurface \(r=r_0\), one has
\[
\widehat{\lambda}\circ J=\rd r.
\]

We consider Hamiltonians \(H\) that are \(C^2\)-small perturbations of Hamiltonians which are zero on \(W\) and linear of slope \(a \in \R_{>0}\) on \(\partial W\times (1,\infty)\), where \(a\) is not a period of any Reeb orbit of \(R_{\lambda}\) on \((\partial W,\lambda|_{\partial W})\). We may assume that \(H\) is nondegenerate. With these conventions in place, the periodic orbits of $H$ are either
\begin{enumerate}[(i)]
\item\label{i} constant orbits in \(W\), with \(\cA_H\approx 0\) or
\item\label{ii} non-constant orbits near Reeb orbits of \(R_\lambda\) on
\(\partial W\), with action close to minus the action of the corresponding Reeb orbit.
\end{enumerate}

In particular, orbits of the latter type may be assumed to have symplectic action \(\cA_H\in (-a,0)\). To see this, note that \(H\) is an \(\bbS^1\)-dependent \(C^2\)-small perturbation of an autonomous Hamiltonian \(h(r)\) with \(h'(r)=a\) for \(r>1+\epsilon\), where \(\epsilon>0\) is small, and \(h(r)=0\) for \(r\le 1\). The non-constant orbits of \(h(r)\) come in \(\bbS^1\)-families of the form
\[
(\gamma(h'(r_0)t),r_0),
\]
where \(\gamma\) is a Reeb orbit of \(R_\lambda\) on \(\partial W\) with action
\(h'(r_0)\), for \(1<r_0<1+\epsilon\). These \(\bbS^1\)-families arise from
reparameterizations of \(\gamma\). The symplectic action of such an orbit is
\[
-h'(r_0)r_0+h(r_0).
\]
When \(\epsilon\ll 1\), this action is approximately the negative of the action of \(\gamma\), and hence lies in \((-a,0)\). By \cite[Lemma 3.3]{BO_Duke}, a \(C^2\)-small perturbation of \(h(r)\) breaks each \(\bbS^1\)-family into two nondegenerate orbits whose actions are arbitrarily close to that of the original family.

After fixing an \(\bbS^1\)-dependent compatible almost complex structure \(J\) that is cylindrically convex near a slice on which the Hamiltonian is linear of slope \(a\), we consider compactified moduli spaces of Floer cylinders:
\[
\cM_{x,y}
=
\overline{
\left\{
u:\R_s\times \bbS^1_t\to \widehat{W}
\left|
\begin{array}{l}
\partial_su+J(\partial_tu-X_H)=0,\\
\displaystyle\lim_{s\to\infty}u(s,\cdot)=x,\\
\displaystyle\lim_{s\to-\infty}u(s,\cdot)=y
\end{array}
\right.
\right\}/\R}.
\]
Here \(x\) and \(y\) are Hamiltonian orbits, and the quotient is by translation in the \(s\)-coordinate.

For generic \(J\), the count of rigid Floer cylinders defines a cochain complex \(C(H)\). The complex is additively generated by the Hamiltonian orbits, with coefficients in any commutative ring \(R\). In this paper, the default coefficient ring is \(R=\C\). Explicitly, the differential is
\[
\delta(x)
=
\sum_{y,\ \virdim \cM_{x,y}=0}
\left(\#\cM_{x,y}\right)y.
\]
The differential depends on choices of coherent orientations of moduli spaces, which we ignore in our exposition. The cochain complex is graded by \(N-\mu_{\CZ}\), where $\mu_{\CZ}$ denotes the Conley--Zehnder index. In general, this complex has \(\Z/2\Z\)-valued degree grading. If \(c_1(W)=0\), then after fixing a global trivialization of \(\det_{\C}TW\), the grading is \(\Z\)-valued.

The orbits of type \ref{i} form a subcomplex \(C_0(H)\) of \(C(H)\), which is naturally identified with the Morse cohomology of $H|_{W}$. Hence we may naturally identify
\begin{equation*}
H C_{0}(H) = H^{\ast}(W;R).
\end{equation*}
The orbits of type \ref{ii} span an $R$-linear subspace of $C(H)$ (which is, in general, not closed under the differential $\delta$). Hence we can view such orbits as generators of the quotient complex
\begin{equation*}
C_+(H) = C(H)/C_{0}(H).
\end{equation*}

We define the filtered symplectic cohomology and positive symplectic cohomology groups
\begin{equation*}
SH^{*,<a}(W) = H^*(C(H_a)), \quad SH^{*,<a}_+(W) = H^*(C_+(H_a)).
\end{equation*}
These groups are independent of the chosen \(H_a\), provided \(a\) is not the action of a Reeb orbit on \(\partial W\); see \cite[Proposition 2.8]{Zhou:RP2n-1}. Such groups do of course depend on the contact form $\lambda|_{\partial W}$ which we suppress from our notation. Since the non-constant Hamiltonian orbits arise from Reeb orbits as in \ref{ii}, we will say informally that \(SH^{*,<a}_+(W)\) is generated by Reeb orbits of action less than \(a\).

Given two Hamiltonians \(H_a,H_b\) with slopes \(a<b\), one obtains continuation
maps by counting solutions to the \(s\)-dependent Floer equation for a
nonincreasing homotopy from \(H_b\) to \(H_a\). These maps induce morphisms
\[
C(H_a)\longrightarrow C(H_b),
\]
and they are compatible with the splitting into zero and positive complexes. The symplectic cohomology and positive symplectic cohomology of \(W\) are then defined by
\[
SH^*(W):=\varinjlim_{a\to \infty} H^*(C(H_a)),
\qquad
SH_+^*(W):=\varinjlim_{a\to \infty} H^*(C_+(H_a)).
\]
These direct limits fit into the tautological exact sequence
\begin{equation}\label{eqn:SHexactsequence}
\ldots \to H^*(W)\to SH^*(W)\to SH^*_+(W)\to H^{*+1}(W)\to \ldots.
\end{equation}

\subsubsection{Setting the Hamiltonian to zero on \(W\)}\label{ss:3.1.1}

In some situations, for instance in neck-stretching arguments, it is convenient to use Hamiltonians that vanish on \(W\). In this case the Hamiltonian orbits consist of non-constant orbits on the cylindrical end \(\partial W\times (1,\infty)\), which correspond to Reeb orbits on \(\partial W\), together with constant orbits at points of \(W\). To handle the latter, fix an auxiliary Morse function \(f\) on \(W\) such that \(\partial_r f>0\) near \(\partial W\).

To define and compute symplectic cohomology using such data, consider a cochain complex \(C(H,f)\) generated by the non-constant Hamiltonian orbits of \(H\) and the critical points of \(f\). The differential counts points in the compactified moduli spaces
\begin{eqnarray}
\cM_{x,y}
&:=&
\overline{
\left\{
u:\R_s\times\bbS^1_t\to\widehat{W}
\left|
\begin{array}{l}
\partial_su+J(\partial_tu-X_H)=0,\\
\displaystyle\lim_{s\to\infty}u=x,\\
\displaystyle\lim_{s\to-\infty}u=y
\end{array}
\right.
\right\}/\R};
\label{eqn:1}\\
\cM_{p,q}
&:=&
\overline{
\left\{
\gamma:\R_s\to W
\left|
\begin{array}{l}
\partial_s\gamma+\nabla_g f=0,\\
\displaystyle\lim_{s\to\infty}\gamma=p,\\
\displaystyle\lim_{s\to-\infty}\gamma=q
\end{array}
\right.
\right\}/\R};
\label{eqn:2}\\
\cM_{x,q}
&:=&
\overline{
\left\{
\begin{array}{l}
u:\C\to \widehat{W},\\
\gamma:(-\infty,0]\to W
\end{array}
\left|
\begin{array}{l}
\partial_su+J(\partial_tu-X_H)=0,\\
\partial_s\gamma+\nabla_g f=0,\\
\displaystyle\lim_{s\to-\infty}\gamma=q,\\
u(0)=\gamma(0),\\
\displaystyle\lim_{s\to\infty}u=x
\end{array}
\right.
\right\}/\R}.
\label{eqn:3}
\end{eqnarray}
Here \(x,y\) are non-constant Hamiltonian orbits, \(p,q\) are critical points of \(f\), and \(g\) is a metric on \(W\) such that \((f,g)\) is a Morse--Smale pair. The relevant Fredholm theory, compactness, and regularity for these moduli spaces are discussed in \cite[\S 2.1]{Zhou:ADC}.

The positive symplectic cochain complex \(C_+(H,f)\) is generated by the non-constant Hamiltonian orbits of \(H\), with differential
\[
\delta_+(x)
=
\sum_{y,\ \virdim \cM_{x,y}=0}
\left(\#\cM_{x,y}\right)y.
\]
The zero symplectic cochain complex \(C_0(H,f)\) is generated by the critical points of the Morse function \(f \in C^{\infty}(W)\), with differential counting \(\cM_{p,q}\); hence it is the Morse cochain complex of \((f,g)\). Finally, the total symplectic cochain complex \(C(H,f)\) is the mapping cone of the cochain map
\(\delta_{+,0}:C_+(H,f)\to C_0(H,f)\), where
\[
\delta_{+,0}(x)
=
\sum_{p,\ \virdim \cM_{x,p}=0}
\left(\#\cM_{x,p}\right)p.
\]
Continuation maps can also be defined in this model, giving an alternative construction of filtered and positive symplectic cohomology \cite[Proposition 2.10]{Zhou:ADC}.

\subsection{Wrapped Floer cohomology}\label{SS3.2}

There is an analogous construction for exact Lagrangians with cylindrical ends,
namely wrapped Floer cohomology. Let
\begin{equation*}
L\subset(W,\lambda)
\end{equation*}
be an exact Lagrangian whose boundary is Legendrian in $\partial W$ and for which the Liouville vector field is tangent to $L$ along a collar neighborhood of its boundary.

To simplify the action functional needed to set up the wrapped Floer cohomology of $L$, we make the following assumptions throughout the paper:
\begin{enumerate}
\item \(\partial L=\Lambda\) is a connected Legendrian in \(\partial W\).
\item If \(\lambda|_L=\rd f\), then \(f=0\) near \(\partial L\).\footnote{By
modifying the Liouville form, one can always assume that \(f\) is constant
near the boundary; see \cite[Lemma 3.1]{Abouzaid-Seidel}. When \(\Lambda\)
is connected, this constant can be taken to be zero. When \(\Lambda\) is not
connected, the same can be achieved after a symplectic/Hamiltonian isotopy
of \(L\); see \cite[Lemma 4.1]{GDR}.}
\item \(L\) is relatively spin, so that the relevant moduli spaces of
holomorphic disks with boundary on \(L\) can be oriented consistently.
\end{enumerate}
The Lagrangians of this article all satisfy these assumptions. Let
\begin{equation*}
\widehat{L} = L \cup (1, \infty) \times \partial L \subset \widehat{W}
\end{equation*}
denote the completion of \(L\) by the Liouville vector field. We extend \(f\) to \(\widehat{L}\) by setting it equal to zero along the cylindrical end.

Let \(H\) be a time-dependent Hamiltonian on \(\widehat{W}\). The symplectic action of a chord \(x:[0,1]\to \widehat{W}\) with \(x(0),x(1)\in \widehat{L}\) is defined to be
\[
\cA_H(x)
=
f(x(1))-f(x(0))
-\int x^*\widehat{\lambda}
+\int_0^1 H\circ x\,\rd t.
\]
If \(H\) is chosen as in the construction of symplectic cohomology with slope
\(a\), then the critical points of \(\cA_H\), i.e. the Hamiltonian chords
satisfying \(x'(t)=X_H\) and \(x(0),x(1)\in \widehat{L}\), are either
\begin{enumerate}
\item critical points of \(H|_L\) on \(L\subset W\) or
\item close to Reeb chords of \(R_\lambda\) on
\(\partial W\).
\end{enumerate}
For chords of the first type, we assume that on a Weinstein neighborhood of \(L\), the Hamiltonian has the form \(H=\pi^*h+\epsilon\rho^2\), where \(\pi:T^*L\to L\), \(h\) is \(C^2\)-small on \(L\), \(\epsilon\ll 1\), and \(\rho\) is the length function associated to a bundle metric on \(T^*L\). Then the critical points of \(h\) are critical points of \(H\) with the same Morse index. These chords have action approximately zero. The actions of the chords of the second type are approximately equal to the negative of the action of their corresponding Reeb chords. In particular, their actions lie in \((-a,0)\).

Fix a time-dependent compatible almost complex structure \(J\) that is cylindrically convex for a Hamiltonian which is linear of slope \(a\). We consider compactified moduli spaces of Floer strips with boundary on
\(\widehat{L}\),
\[
\cM_{x,y}
=
\overline{
\left\{
u:\R_s\times [0,1]_t\to \widehat{W}
\left|
\begin{array}{c}
\partial_su+J(\partial_tu-X_H)=0,\\
u(\cdot,0),u(\cdot,1)\in \widehat{L},\\
\displaystyle\lim_{s\to\infty}u(s,\cdot)=x,\quad
\displaystyle\lim_{s\to-\infty}u(s,\cdot)=y
\end{array}
\right.
\right\}/\R}.
\]
For generic \(J\), the count of rigid Floer strips defines a cochain complex \(C(L,H)\), freely generated over \(R\) by Hamiltonian chords. Its differential is
\[
\delta(x)
=
\sum_{y,\ \virdim \cM_{x,y}=0}
\left(\#\cM_{x,y}\right)y.
\]
Our assumption that $L$ is relatively spin ensures that we can coherently orient our moduli spaces, and therefore use coefficients in rings $R$ containing $\Q$. This condition ensures that our moduli spaces may be coherently oriented.

The wrapped Floer cohomology is \(\Z/2\Z\)-graded when \(L\) is oriented. This grading can be refined to a \(\Z\)-grading, following \cite{Graded}, if the relative first Chern class \(c_1(W,L)\in H^2(W,L;\Z)\) is \(2\)-torsion. Our grading convention follows \cite{Abouzaid-Seidel}, so that
\eqref{eqn:exact_L} has no degree shift and the product on \(HW^*(L)\) has
degree zero. Let \(c\) be a nondegenerate Reeb chord on \(\Lambda\), and let \(\overline{c}\) be the corresponding nondegenerate Hamiltonian chord for a
Hamiltonian \(H=h(r)\) with \(h'(r)>0\). Then
\[
\dim L-|\overline{c}|-2=|c|_{\SFT},
\]
where \(|\overline{c}|\) is the wrapped Floer cohomology grading used here and
\(|c|_{\SFT}\) is the SFT degree of the Reeb chord. By contrast, the wrapped Floer cohomology conventions in \cite{BEE} and in the arXiv version of \cite{Lazarev:flexiblefillings} use \(|c|_{\SFT}\) as the grading. This accounts for the grading discrepancy in \cite[\S 7.1.2, arXiv version]{Lazarev:flexiblefillings}.

As in symplectic cohomology, \(C(L,H)\) has a subcomplex \(C_{0}(L, H)\) generated by constant chords, whose cohomology computes \(H^*(L)\). We also have a quotient complex 
\[
C_+(L,H) = C(L, H)/C_{0}(L, H)
\]
generated by non-constant chords.

We define the filtered wrapped Floer cohomology and positive wrapped Floer cohomology as 
\[
HW^{*,<a}(L) = H^*(C(L,H_a)), \quad HW^{*,<a}_+(L) = H^*(C_+(L,H_a)).
\]
These groups are independent of \(H_a\), provided \(a\) is not the action of any Reeb chord of \(\Lambda\) on \(\partial W\). We will say informally that \(HW^{*,<a}_+(L)\) is generated by Reeb chords of action less than \(a\).

Continuation maps from a Hamiltonian of slope \(a\) to another of slope \(b>a\) are defined as in symplectic cohomology cochain complexes and respect subcomplexes of constant chords. The wrapped Floer cohomology and positive wrapped Floer cohomology are then defined as
\[
HW^*(L):=\varinjlim_{a\to \infty}H^*(C(L,H_a)),
\qquad
HW_+^*(L):=\varinjlim_{a\to \infty}H^*(C_+(L,H_a)).
\]
These limits fit into the tautological exact sequence
\begin{equation}\label{eqn:exact_L}
\ldots \to H^*(L)\to HW^*(L)\to HW^*_+(L)\to H^{*+1}(L)\to \ldots.
\end{equation}

As in symplectic cohomology, one can take \(H=0\) on \(W\) and use an auxiliary Morse function \(h\) on \(L\), with \(\partial_r h>0\) near \(\partial L\), to set up wrapped Floer cohomology. This gives a model of wrapped Floer cohomology whose generators are critical points of \(h\) together with non-constant Hamiltonian chords. The differential is defined by counting analogues of \eqref{eqn:1}--\eqref{eqn:3}; the analogue of \eqref{eqn:3} is a Floer strip with a removable boundary puncture followed by a gradient trajectory of \(h\). The analytic setup is analogous to \cite[\S 2.1]{Zhou:ADC}. We denote the resulting complexes by
\[
C(L,H,h),\qquad C_0(L,H,h),\qquad C_+(L,H,h).
\]

\begin{proposition}\label{prop:Euler}
Let \(L\) be an oriented relatively spin exact Lagrangian filling of
\(\Lambda\subset(\bbS^{2n+1},\xi_{\std})\) in a Liouville filling \(W\) of
\((\bbS^{2n+1},\xi_{\std})\). We suppose that $\Lambda$ and the contact form $\alpha$ on $\bbS^{2n+1}$ used to set up our Floer cohomologies are as in \Cref{prop:contact}. If \(1\le a<9\pi-1\), then
\[
\chi\bigl(HW^{*,<a}_+(L)\bigr)
=
(-1)^{\frac{n(n-1)}{2}}\tb(\Lambda).
\]
\end{proposition}

\begin{proof}
The Euler characteristic of \(HW^{*,<a}_+(L)\) is equal to the Euler characteristic of the cochain complex \(C_+(L,H_a)\). This complex is generated by non-constant Reeb chords of action less than \(a\). According to \Cref{prop:contact}, these are precisely the Reeb chords of action less than \(1\). Such Reeb chords are in one-to-one correspondence with self-intersections of the Lagrangian projection of \(\Lambda\).

By \cite[Proposition 3.3]{EES},
\[
\tb(\Lambda)
=
(-1)^{\frac{(n-2)(n-1)}{2}}
\sum_c(-1)^{|c|_{\SFT}}.
\]
Since \(\dim L=n+1\), the grading relation above gives
\[
|c|_{\SFT}=n-|\overline{c}|-1.
\]
Therefore
\[
\chi\bigl(HW^{*,<a}_+(L)\bigr)
=
\sum_c(-1)^{|\overline{c}|}
=
(-1)^{\frac{n(n-1)}{2}}\tb(\Lambda).
\]
\end{proof}

\subsection{Structures on symplectic cohomology and wrapped Floer cohomology}\label{ss:3.3}

Symplectic cohomology and wrapped Floer cohomology carry several algebraic structures defined by counting suitable holomorphic curves. See, for example, \cite{Ritter1}. A convenient way to write the relevant Floer equations is
\begin{equation}\label{eqn:floerTQFT}
u:(S,\partial S)\to(\widehat{W},\widehat{L}),
\qquad
(\rd u-X_H\otimes \beta)^{0,1}=0.
\end{equation}
Here \(H\) is a Hamiltonian on \(\widehat{W}\) of slope \(1\), \(S\) is a fixed punctured Riemann surface, possibly with boundary, and \(\beta\) is a one-form on \(S\). The surface \(S\) may have both interior and boundary punctures, equipped with cylindrical coordinates and asymptotic to Hamiltonian orbits and Hamiltonian chords, respectively. We require
\[
\rd\beta\le 0,\qquad \beta|_{\partial S}=0.
\]
Along a neighborhood of each interior or boundary puncture, modeled on
\((0,\pm\infty)_s\times \bbS^1_t\) or
\((0,\pm\infty)_s\times [0,1]_t\), respectively, we require
\[
\beta=a\,\rd t
\]
for an appropriate weight \(a\). The cylindrical ends are part of the geometric data of \(S\). In particular, the parameterization of a cylindrical end of $S$ along interior punctures determine asymptotic markers, that is, tangent vectors at the punctures (viewed as marked points in a compactification of $S$) defined up to positive rescaling.

The almost complex structure on \(\widehat{W}\) used to define Equation \eqref{eqn:floerTQFT} may depend on the point of \(S\), which is useful for ensuring somewhere-injectivity and regularity of open strata of moduli spaces.\footnote{In general, \(H\) may also depend on \(S\); we omit this for simplicity. For continuation maps, if one wants to preserve the \(C^2\)-smallness of the Hamiltonian on \(W\), then \(H\) may also depend on the \(s\)-coordinate.} For example, for the Floer differential on \(SH^{*,<a}(W)\), one takes \(S=\R_s\times \bbS^1_t\), \(\beta=a\,\rd t\), and \(J\) depending only on \(t\), so as to preserve translation invariance in the \(s\)-coordinate. For continuation maps, one replaces \(\beta\) by \(f(s)\,\rd t\), with \(f'(s)\le 0\) and prescribed asymptotic slopes. The wrapped Floer case is analogous.

Consider the thrice-punctured sphere
\[
\bbS_{2,1}:=\CP^1\setminus \{0,1,\infty\}.
\]
We view \(0\) and \(1\) as positive punctures and \(\infty\) as a negative puncture and fix cylindrical coordinates near the punctures. Choose \(\beta\) on \(\bbS_{2,1}\) such that \(\rd\beta\le 0\), with
\[
\beta=a\,\rd t,\quad b\,\rd t,\quad c\,\rd t
\]
near \(0,1,\infty\), respectively. Such a form exists if and only if \(c\ge a+b\). The pair-of-pants product on symplectic cohomology is defined by counting rigid solutions of
\[
u:\bbS_{2,1}\to\widehat{W},
\qquad
(\rd u-X_H\otimes \beta)^{0,1}=0.
\]
Here \(u\) is asymptotic to Hamiltonian orbits \(x,y,z\) of \(aH,bH,cH\), respectively. The count of such solutions gives a cochain map
\[
C(aH)\otimes C(bH)\to C(cH),
\]
and hence a degree-zero product on filtered cohomology
\[
\smile:
SH^{*,<a}(W)\otimes SH^{*,<b}(W)\to SH^{*,<c}(W).
\]
This product is compatible with continuation maps, and its direct limit is the pair-of-pants product. With this product, the map \(H^*(W)\to SH^*(W)\) in the tautological exact sequence of Equation \eqref{eqn:SHexactsequence} is a morphism of graded supercommutative rings. The resulting maps on cohomology do not depend on the auxiliary choices of \(\bbS_{2,1}\), its cylindrical ends, or \(\beta\).

There is an analogous pair-of-pants product in wrapped Floer cohomology. Counting solutions of
\[
u:(\D_{2,1},\partial \D_{2,1})\to(\widehat{W},\widehat{L}),
\qquad
(\rd u-X_H\otimes \beta)^{0,1}=0,
\]
where \(\D_{2,1}\) is the unit disk with three fixed boundary punctures, two of which are positive, defines the product
\[
\mu_2:
HW^{*,<a}(L)\otimes HW^{*,<b}(L)\to HW^{*,<a+b}(L).
\]
After passing to the direct limit, the product $\mu_2$ gives \(HW^*(L)\) the structure of a ring with unit, which is non-commutative in general. The map \(H^*(L)\to HW^*(L)\) from the
tautological exact sequence of Equation \eqref{eqn:exact_L} is a unital ring map.

Let \(\D_{1,1|1,0}\) denote the unit disk with two fixed boundary punctures (one positive and one negative), and one fixed interior puncture (which we view as positive), all equipped with cylindrical local coordinates. Choose \(\beta\) such that 
\[
\rd\beta\le 0, \quad \beta|_{\partial \D_{1,1|1,0}}=0,
\]
\[
\beta=a\,\rd t
\quad\text{near the positive boundary puncture,}
\]
\[
\beta=b\,\rd t
\quad\text{near the positive interior puncture,}
\]
\[
\beta=c\,\rd t
\quad\text{near the negative boundary puncture,}
\]
with \(c\ge a+b\). Counting rigid solutions of
\[
u:(\D_{1,1|1,0},\partial \D_{1,1|1,0})
\to(\widehat{W},\widehat{L}),
\qquad
(\rd u-X_H\otimes \beta)^{0,1}=0,
\]
where \(u\) is asymptotic to a Hamiltonian chord \(x\) of \(aH\) at the positive
boundary puncture, to a Hamiltonian orbit \(y\) of \(bH\) at the positive
interior puncture, and to a Hamiltonian chord \(z\) of \(cH\) at the negative
boundary puncture, yields an operation
\[
m:
SH^{*,<b}(W)\otimes HW^{*,<a}(L)\to HW^{*,<c}(L).
\]
Passing to the direct limit, this operation gives \(HW^*(L)\) the structure of an \(SH^*(W)\)-module.

\begin{remark}
The construction of the operations \(\smile\), \(\mu_2\), and \(m\) requires the use of \(C^2\)-small Hamiltonians on the domain \(W\). Thus, for these product-type structures, one cannot simply use a mixed Hamiltonian-Floer cascade model with \(H=0\) on \(W\) to handle the Morse--Bott degeneracy of constant orbits. The issue is that \(H=0\) is not Morse--Bott nondegenerate along \(\partial W\). Since stable manifolds of critical points are contained in the interior of \(W\), the Fredholm setup breaks down when the only outputs from the zero part are critical points, unless all inputs and outputs are critical points; see \cite[\S 2.1]{Zhou:ADC}. However, when one input is a critical point and the other is a non-constant orbit or chord, the unstable manifolds of critical points may reach the boundary, creating analytic difficulties in setting up the moduli spaces.\footnote{It is nevertheless possible to build an \(H^*(W)\)-module structure on \(SH^*(W)\) or \(HW^*(L)\) using a cascade model by treating the input from \(H^*(W)\) as a constraint at a marked point.} This is not merely a technical issue: some related constructions should not carry a unital ring structure; see \Cref{rmk:ring}.
\end{remark}

Let \(\D_{0,1|1,0}\) denote the unit disk with one fixed negative boundary
puncture and one fixed positive interior puncture. Choose \(\beta\) such that
\[
\rd\beta\le 0,\qquad
\beta|_{\partial \D_{0,1|1,0}}=0,
\]
\[
\beta=a\,\rd t \quad \text{near the interior puncture},
\]
\[
\beta=b\,\rd t \quad \text{near the boundary puncture},
\]
with \(b\ge a\). Consider solutions to
\[
u:(\D_{0,1|1,0},\partial \D_{0,1|1,0})
\to(\widehat{W},\widehat{L}),
\qquad
(\rd u-X_H\otimes \beta)^{0,1}=0,
\]
where \(u\) is required to be asymptotic to a Hamiltonian orbit \(x\) at the interior puncture and to a Hamiltonian chord \(y\) at the boundary puncture. The count of such rigid solutions gives a map
\[
\CO:SH^{*,<a}(W)\to HW^{*,<b}(L).
\]
By a slight abuse of terminology, we call this the \emph{closed-open map}.

Unlike the product structures above, \(\CO\) can be defined using a Hamiltonian that vanishes on \(W\) as follows: Fix a Morse function \(f\) on \(W\) such that \(\partial_r f>0\) along \(\partial W\). In a Weinstein neighborhood of \(L\), assume
\[
f=\pi^*h+\epsilon\rho^2.
\]
Here \(\epsilon>0\), \(\pi:T^*L\to L\), and \(\rho\) is the length function associated to a bundle metric on \(T^*L\). Using a suitable metric on \(W\) with a similar splitting near \(L\), the critical points on \(L\) provide generators spanning a quotient of the Morse cochain complex on \(W\), canonically identified with the Morse cochain complex of \(h\) on \(L\). This quotient map induces the restriction map \(H^*(W)\to H^*(L)\).

We define compactified moduli spaces
\begin{equation}\label{eqn:CO1}
\cM^{\CO}_{x,y}
:=
\overline{
\left\{
u:(\D_{0,1|1,0},\partial\D_{0,1|1,0})
\to(\widehat{W},\widehat{L})
\left|
\begin{array}{c}
(\rd u-X_H\otimes \beta)^{0,1}=0,\\
u(\partial\D_{0,1|1,0})\subset\widehat{L},\\
\displaystyle\lim_{s\to\infty}u=x,\quad
\displaystyle\lim_{s\to-\infty}u=y
\end{array}
\right.
\right\}/\R}.
\end{equation}
Here \(x\) is a non-constant Hamiltonian orbit of \(aH\), \(y\) is a non-constant Hamiltonian chord of \(bH\), and \(s\) is the cylindrical coordinate defined near the punctures. We also define
\begin{equation}
\cM^{\CO}_{x,q}
:=
\overline{
\left\{
\begin{array}{l}
u:(\D_{0,1|1,0},\partial\D_{0,1|1,0})
\to(\widehat{W},\widehat{L}),\\
\gamma:(-\infty,0]\to L
\end{array}
\left|
\begin{array}{l}
(\rd u-X_H\otimes \beta)^{0,1}=0,\\
\partial_s\gamma+\nabla h=0,\\
u(\partial\D_{0,1|1,0})\subset \widehat{L},\\
\displaystyle\lim_{s\to-\infty}\gamma=q,\\
\displaystyle\lim_{s\to-\infty}u=\gamma(0),\\
\displaystyle\lim_{s\to\infty}u=x
\end{array}
\right.
\right\}/\R}.
\label{eqn:CO3}
\end{equation}
Here \(x\) is a non-constant Hamiltonian orbit of \(aH\), and \(q\) is a critical point of \(h\) on \(L\). Since the stable manifold of \(q\) is precompact in the interior of \(L\) (where the Hamiltonian vanishes), the negative boundary puncture of \(u\) is removable, and the condition
\(\lim_{s\to-\infty}u=\gamma(0)\) is well-defined.

We define a map
\[
\CO:C(H,f)\to C(L,H,h)
\]
using the mixed Hamiltonian--Morse models as follows:
\begin{enumerate}
\item Non-constant Hamiltonian orbits $x$ are sent to non-constant chords $y$ by counting points in \(\cM^{\CO}_{x,y}\).
\item Morse critical points in $W$ are sent to Morse critical points in $L$ via the quotient of Morse cochain complexes.
\item Non-constant Hamiltonian orbits $x$ are sent to Morse critical points $q$ of $L$ by counting points in \(\cM^{\CO}_{x,q}\).
\end{enumerate}

Such counts define morphisms fitting into the commutative diagram
\begin{equation}\label{eqn:CO}
\xymatrix{
\ldots \ar[r] &
H^*(W) \ar[d]^{\CO_0} \ar[r] &
SH^{*,<a}(W) \ar[d]^{\CO} \ar[r] &
SH^{*,<a}_+(W) \ar[d]^{\CO_+} \ar[r] &
H^{*+1}(W) \ar[d] \ar[r] &
\ldots\\
\ldots \ar[r] &
H^*(L) \ar[r] &
HW^{*,<a}(L) \ar[r] &
HW^{*,<a}_+(L) \ar[r] &
H^{*+1}(L) \ar[r] &
\ldots.
}
\end{equation}
To see that the square on the right commutes, note that counting points in \(\cM^{\CO}_{x,q}\) determines a morphism
\[
\CO_{+,0}:C_+(H,f)\to C_0(L,H,h)
\]
giving the desired chain-level homotopy.

\begin{proposition}\label{prop:moduleStructureDecomposition}
The \(SH^*(W)\)-module structure on \(HW^*(L)\) is given by composing the
closed-open map with the product on \(HW^*(L)\). More precisely, for
\(a,b>0\), the filtered module action agrees with
\[
SH^{*,<b}(W)\otimes HW^{*,<a}(L)
\xrightarrow{\CO\otimes \mathrm{id}}
HW^{*,<b}(L)\otimes HW^{*,<a}(L)
\xrightarrow{\mu_2}
HW^{*,<a+b}(L).
\]
\end{proposition}

\begin{proof}
This follows from a one-parameter deformation of the domain curve, together
with the associated geometric data and Floer equations, in which the interior
puncture is pushed to the boundary; see \Cref{fig:deg}.
\end{proof}

\begin{figure}[htb]
{\small
\begin{overpic}[scale=0.4]{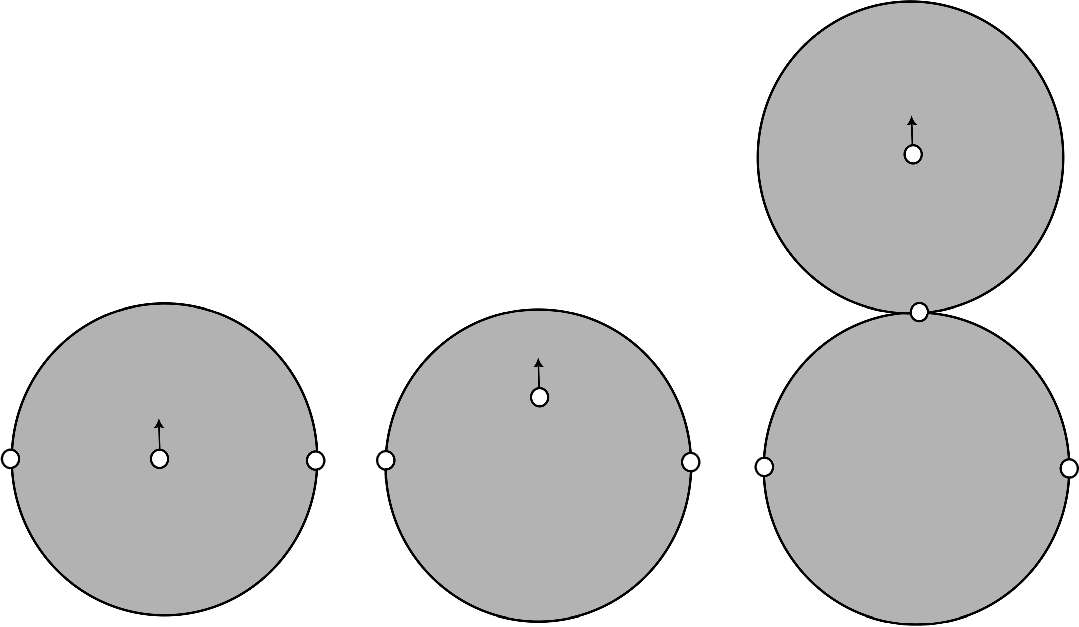}
\put (30, 20) {$m$}
\put (102, 25) {$m$}
\put (175,30) {$\mu_2$}
\put (173,75) {$\CO$}
\put (4,30) {$+$}
\put (33,30) {$+$}
\put (51,30) {$-$}
\put (77,30) {$+$}
\put (123,30) {$-$}
\put (106,42) {$+$}
\put (150,30) {$+$}
\put (197,30) {$-$}
\put (173,52) {$+$}
\put (173,62) {$-$}
\put (178,90) {$+$}
\end{overpic}}
\caption{Degeneration of punctured surfaces as interior marked points tend to the boundary of disks as described in the proof of \Cref{prop:moduleStructureDecomposition}.}
\label{fig:deg}
\end{figure}

We are now ready to prove a weak version of \Cref{thm:vanishingCH} using the aforementioned algebraic structures.

\begin{corollary}\label{cor:nofilling}
Under the assumptions of \Cref{thm:vanishingCH}, \((\bbS^{2n+1}_{+1}(\Lambda), \xi_{+1}(\Lambda))\) does not admit a Liouville filling.
\end{corollary}

\begin{proof}
We argue by contradiction. Suppose that \((\bbS^{2n+1}_{+1}(\Lambda), \xi_{+1}(\Lambda))\) admits a Liouville
filling \(V\). Then \(V\cup W_{\Lambda}\) is a Liouville filling of \((\bbS^{2n+1},\xi_{\std})\). Let
\[
L\subset  V\cup W_\Lambda, \quad \partial L = \Lambda
\]
be the core Lagrangian disk of the surgery cobordism.

We assume that the contact form $\alpha$ on $\bbS^{2n+1}$ is as described in \Cref{prop:contact}. Then there are no Reeb chords of \(\Lambda\) with action in \([1,4\pi+\delta)\), for \(0<\delta\ll 1\). Hence
\[
HW^{*,<4\pi+\delta}(L)=HW^{*,<1}(L).
\]
While the filtered wrapped Floer cohomology product $\mu_2$ a priori maps \(HW^{*,<1}(L)\otimes HW^{*,<1}(L)\) to \(HW^{*,<2}(L)\), the action gap for Reeb chords  identifies \(HW^{*,<2}(L)\) with \(HW^{*,<1}(L)\). Thus the pair-of-pants product endows \(HW^{*,<1}(L)\) with the structure of a unital ring.

By \cite[Proposition 3.2]{product}, the map
\[
H^0(V\cup W_\Lambda)\longrightarrow
SH^{*,<4\pi+\delta}(V\cup W_\Lambda)
\]
is zero. Using the closed-open map of Equation  \eqref{eqn:CO}, the image of the unit \(1\in H^0(L)\) in \(HW^{*,<4\pi+\delta}(L)=HW^{*,<1}(L)\) is therefore zero. Since this image is the unit of the ring \(HW^{*,<1}(L)\), we conclude that
\[
HW^{*,<1}(L)=0.
\]
The tautological exact sequence of Equation \eqref{eqn:exact_L} then gives
\[
\chi\bigl(HW^{*,<1}_+(L)\bigr)=-\chi(H^{\ast}(L)) = -\chi(L) =-1,
\]
as \(L\) is a disk. On the other hand, by \Cref{prop:Euler},
\[
\chi\bigl(HW^{*,<1}_+(L)\bigr)
=
(-1)^{\frac{n(n-1)}{2}}\tb(\Lambda).
\]
Therefore
\[
\tb(\Lambda)=(-1)^{\frac{n(n-1)}{2}+1},
\]
contradicting the hypothesis.
\end{proof}

In dimension \(3\), Conway--Etnyre--Tosun \cite{CET} showed that \((\bbS^3_{+1}(\Lambda), \xi_{+1}(\Lambda))\) admits a Liouville filling if and only if \(\Lambda\) admits an exact Lagrangian disk filling, thereby implying \(\tb(\Lambda)=-1\). By working slightly harder and using the exactness of the surgery cobordism, one can also rule out strong fillings in the statement of \Cref{cor:nofilling}.\footnote{When \(\tb(\Lambda)\ne -1\), \(\bbS^3_{+1}(\Lambda)\) is a rational homology sphere, so there is no distinction between weak and strong fillings for this manifold.} The proof of \Cref{thm:vanishingCH} replaces geometric fillings by suitable algebraic fillings.

Since the proof of \Cref{cor:nofilling} uses the map \(\CO\) from \(SH^*(W)\), whose defining curves have only one positive interior puncture, the elements of the relevant moduli spaces obtained by neck-stretching along \((\bbS^{2n+1}_{+1}(\Lambda),\xi_{+1}(\Lambda))\) may degenerate into SFT buildings containing holomorphic planes in the hypothetical filling. To manage this bubbling of planes algebraically, we considers augmentation of the contact homology CDGAs, which we now recall.

\subsection{Contact homology and augmentations}\label{ss:CHreview}

We briefly recall contact homology \cite{Bao-Honda, EliGivHof, Pardon:Contact} and
augmentations. Let \((Y^{2N-1},\xi)\) be a contact manifold. Fix a contact form \(\alpha\) such that all Reeb orbits are nondegenerate. For each good Reeb orbit \(\gamma\), associate a formal variable \(q_\gamma\).\footnote{We defer to the citations for the definitions of good and bad Reeb orbits.} Let $\bm{k}$ be a field containing $\Q$, and let \(V\) denote the \(\Z/2\Z\)-graded \(\bm{k}\)-vector space generated by the \(q_\gamma\), with grading
\[
|q_\gamma|=\mu_{\CZ}(\gamma)+N-3.
\]
Counting rigid punctured holomorphic curves in the symplectization of $(Y, \alpha)$ with genus zero, one positive puncture, and any number of negative punctures gives degree \(-1\) operations
\[
\partial_k:V\to S^{k} V=(\otimes^{k} V)/\Sigma_k.
\]
Here $\Sigma_{k}$ is the group of permutations on $k$ letters. The sum
\begin{equation*}
\partial=\sum_{k=0}^{\infty}\partial_k,
\end{equation*}
extended to the graded-commutative tensor algebra
\[
SV=\bigoplus_{k=0}^{\infty}S^kV
\]
by the Leibniz rule
\[
\partial(ab)=\partial(a)b+(-1)^{|a|}a\partial(b),
\]
gives \((SV,\partial)\) the structure of a semi-free commutative differential graded algebra (CDGA). The CDGAs of contact homology are \emph{finitely filtered}, meaning that the vector space $V$ admits a filtration
\begin{equation*}
V^{0} \subset V^{1} \subset \cdots \subset V, \quad \cup_{i} V^{i} = V
\end{equation*}
so that for all $i$,
\begin{equation*}
\dim V^{i} < \infty, \quad \partial V^{i+1} \subset SV^{i}.
\end{equation*}
In practice we can take $V^{i}$ to be generated by the Reeb orbits of action $\leq L_{i}$ where $L_{i}$ are the actions of closed Reeb orbits, ordered so that $L_{i+1} > L_{i}$. Due to the presumed non-degeneracy of the contact form, each $V^{i}$ is then finitely generated.

More precisely, fix a \(\rd(e^s\alpha)\)-compatible \(s\)-invariant almost complex structure on \(\R_s\times Y\) such that
\[
\alpha\circ J=\rd s,
\]
i.e. \(J\) is cylindrically convex. For each simple Reeb orbit, fix a base point on its image. This determines a constant-speed parameterization of the Reeb orbit. Let \(\gamma\) be a good Reeb orbit, and let
\[
\Gamma=\{\gamma_1,\ldots,\gamma_m\}
\]
be an ordered multiset of good Reeb orbits. Define \(\cM_{\gamma,\Gamma}\) to be
the compactified moduli space
\[
\overline{
\left\{
\begin{array}{c}
u:\CP^1\setminus\{z_+,z_{1,-},\ldots,z_{m,-}\}\to \R_s\times Y\\
z_+,z_{1,-},\ldots,z_{m,-}\in\CP^1 \text{ pairwise distinct}\\
\theta_+\in ST_{z_+}\CP^1,\quad
\theta_{i,-}\in ST_{z_{i,-}}\CP^1
\end{array}
\left|
\begin{array}{c}
(\rd u)^{0,1}=0,\\
u(z_+)=\gamma,\\
u(z_{i,-})=\gamma_i
\end{array}
\right.
\right\}
/\mathrm{Aut}\times \R
}.
\]
Here \(u(z_+)=\gamma\) means that \(u\) converges to \((+\infty,\gamma)\) at the positive puncture \(z_+\), and that the asymptotic marker \(\theta_+\) maps to the chosen base point on the image of \(\gamma\). Similarly, \(u(z_{i,-})=\gamma_i\) means convergence to \((-\infty,\gamma_i)\) at the negative puncture \(z_{i,-}\), with the analogous condition on the asymptotic marker. The automorphism group consists of biholomorphisms of \(\CP^1\) that intertwine all of this data, and the \(\R\)-action is induced by translation in the symplectization coordinate.

The operation \(\partial_k\) is defined by
\[
\partial_k(q_\gamma)
=
\sum_{[\Gamma],\ |\Gamma|=k}
\frac{1}{\mu_\Gamma\kappa_\Gamma}
\#\cM_{\gamma,\Gamma}\,q^\Gamma,
\]
where
\[
q^\Gamma:=q_{\gamma_1}\cdots q_{\gamma_m}.
\]
The sum is over unordered multisets \([\Gamma]\) with \(k\) elements such that \(\cM_{\gamma,\Gamma}\) has virtual dimension zero. Here \(\Gamma\) is an ordered representative, \(\kappa_\Gamma\) is the product of the covering multiplicities of the orbits in \(\Gamma\), and \(\mu_\Gamma\) is the order of the isotropy group of \(\Gamma\) under reordering. Thus, if
\[
\Gamma=
\{\underbrace{\gamma_1,\ldots,\gamma_1}_{i_1},
\ldots,
\underbrace{\gamma_k,\ldots,\gamma_k}_{i_k}\},
\qquad \gamma_i\ne\gamma_j \text{ for } i\ne j,
\]
then
\[
\mu_\Gamma=i_1!\cdots i_k!.
\]
The count \(\#\cM_{\gamma,\Gamma}\) requires a choice of virtual perturbation scheme and associated auxiliary data, as in \cite{Bao-Honda,Pardon:Contact}. The finite filtration condition follows from the facts (1) that there are only finitely-many closed Reeb orbits under a given action bound and (2) that perturbed holomorphic maps $u$ contributing to the differential always have non-negative contact area, $\int u^{\ast}\rd\alpha \geq 0$.

We denote this CDGA by \(\CC(Y,\alpha,\Theta)\), where \(\Theta\) represents the
auxiliary data required by the chosen virtual setup. Its homology is called the
contact homology of \(Y\), and it is a contact invariant
\cite{Bao-Honda,Pardon:Contact}; that is, it depends only on \(\xi\).

\begin{remark}\label{rmk:finite}
Pardon's construction \cite{Pardon:Contact} produces a contact DGA generated by
all good Reeb orbits. In contrast, \cite{Bao-Honda} constructs filtered versions
\(\CC^{L_i}(Y,\alpha_i,\Theta_i)\), generated by Reeb orbits of action at most
\(L_i > 0\), for a sequence of contact forms \(\alpha_i\) that are only \(C^0\)-close
to one another. The cobordism maps
\[
\CC^{L_i}(Y,\alpha_i,\Theta_i)\to \CC^{L_j}(Y,\alpha_j,\Theta_j),
\qquad i<j,
\]
do not form a direct system of DGA maps on the chain level, but they do form a
direct system on homology. The contact homology is then defined as the direct
limit of these homologies. It follows that if the total contact homology vanishes, then the filtered contact homology vanishes for all sufficiently large thresholds. Similarly, the augmented filtered symplectic/wrapped Floer cohomology used in \S \ref{s4} and \ref{s5} requires an augmentation to the filtered contact algebra with a sufficiently high threshold.  Hence for our purposes, working with such filtered
versions is sufficient if one uses the virtual setup of \cite{Bao-Honda}. In
what follows, we suppress this technical distinction.
\end{remark}

If \(c_1(\xi)=0\), then after choosing a trivialization of \(\det_{\C}\xi\), the
Conley--Zehnder index \(\mu_{\CZ}(\gamma)\) is \(\Z\)-valued for every
nondegenerate Reeb orbit \(\gamma\). If the homology class of \(\gamma\) is
torsion, then \(\mu_{\CZ}(\gamma)\) is independent of the chosen trivialization.
In these cases, \(\partial\) has degree \(-1\) when \(q_\gamma\) is graded by
\(\mu_{\CZ}(\gamma)+N-3\in\Z\).

Let $\Algebra = (SV, \partial)$ be a semi-free CDGA. An \emph{augmentation} is a CDGA morphism
\[
\epsilon:\Algebra \to(\bm{k},0).
\]
Given an augmentation \(\epsilon\), define an algebra isomorphism
\(F_{-\epsilon}:SV\to SV\) by
\[
F_{-\epsilon}(x)=x-\epsilon(x),
\]
with inverse \(F_\epsilon\). Under this change of coordinates,
\[
F_\epsilon\circ \partial\circ F_{-\epsilon}
\]
is again induced by operations
\[
\partial_{\epsilon,k}:V\to S^kV,
\]
with the important property that \(\partial_{\epsilon,0}=0\). Consequently,
\(\partial_{\epsilon,1}\) is a differential on \(V\). The corresponding
linearized contact homology is
\[
\LCH_*(Y,\epsilon):=H_*(V,\partial_{\epsilon,1}).
\]
Although \(\LCH_*(Y,\epsilon)\) depends on the augmentation \(\epsilon\), the
collection of linearized contact homologies, as \(\epsilon\) ranges over all
augmentations, is a contact invariant by \cite{chaidez2024contact}. The
functorial foundation of linearized contact homology will be established in
\cite{DG} via chain-level DGA homotopies for contact homology.

Since the contact DGA of the empty set is the ground field \(\bm{k}\), counting
rigid holomorphic planes in a Liouville filling \(W\), together with the relevant
virtual perturbation data, gives an augmentation of the contact DGA by
\cite{Bao-Honda,Pardon:Contact}. If \(c_1(W)=0\), then after fixing a
trivialization of \(\det_{\C}TW\), one obtains a \(\Z\)-graded augmentation of
the contact DGA with respect to the induced trivialization of \(\det_{\C}\xi\). In practice, we take \(\bm{k}=\Q\) or \(\C\). The property \(\CH(Y,\xi)=0\) is independent of this choice.

\Cref{Lemma:aug}, below, will be important in our study of chain-level contact homology algebras. To prove it, we will need the following preliminary result.

\begin{lemma}\label{Lemma:ConstructibleIntersection}
Let $D_{1} \supset D_{2} \supset \cdots$ be a sequence of non-empty subsets of $\C^{N}$ which are constructible in the Zariski topology. Then $\cap_{j} D_{j} \neq \emptyset$.
\end{lemma}

\begin{proof}
Write $\overline{D}_{j}$ for the Zariski closure of $D_{j}$, for any $j$. Since $\C[x_{1}, \dots, x_{N}]$ is Noetherian and the $\overline{D}_{j}$ are closed and nested, there is some $j_{0}$ such that $\overline{D}_{j} = \overline{D}_{j_{0}}$ for $j \geq j_{0}$. Each $D_{j}$ contains a set of the form $\overline{D}_{j} \setminus Y_{j}$ where $Y_{j} \subset \overline{D}_{j}$ is closed and nowhere dense. Hence for $j \geq j_{0}$, the $Y_{j}$ are closed subsets of $\overline{D}_{j_{0}}$ with dense complements in the Zariski topology, and therefore in the standard topology, cf. \cite[Theorem 2.33]{Mumford}. As $\overline{D}_{j_{0}}$ is a closed subset of $\C^{N}$, the Baire category theorem (applied to the complements of the $Y_{j}$ using the standard topology) implies that $\overline{D}_{j_{0}} \setminus \cup_{j \geq j_{0}} Y_{j} \subset \cap_{j} D_{j}$ is non-empty.
\end{proof}

\begin{lemma}\label{Lemma:aug}
Let $\Algebra = (SV, \partial)$ be a finitely filtered, semi-free CDGA with $\C$ coefficients and $\Z/2\Z$ grading. Then $\Algebra$ admits an augmentation if and only if $H\Algebra \neq 0$.
\end{lemma}

\begin{proof}
Clearly the existence of an augmentation implies the non-vanishing of homology. We first work out a finitely generated version of the desired result using the assumed filtration $V^{0} \subset V^{1} \subset \cdots \subset V$.

For each $i$, write $\Algebra^{i} = (SV^{i},\partial)$. Let $V^{i}_{ev}$ and $V^{i}_{odd}$ be the subspaces of $V^{i}$ of even and odd grading, respectively. We have an associated augmentation variety
\begin{equation*}
Z^{i} \subset \C^{N_{i}}, \quad N_{i} = \dim V^{i}_{ev}
\end{equation*}
with coordinate variables indexed by generators $v_{ev, j} \in V^{i}_{ev}$, defined as follows. Let
\begin{equation*}
P: \Algebra^{i} \rightarrow \C[V^{i}_{ev}]
\end{equation*}
be the algebra map which sends each generator $v_{odd} \in V^{i}_{odd}$ to zero, and preserves generators in $V^i_{ev}$. Then let
\begin{equation*}
\mathcal{I}^{i} \subset \C[V^{i}_{ev}]
\end{equation*}
be the ideal generated by the $P(\partial v_{odd})$, and $Z^{i}$ the variety associated to $\mathcal{I}^{i}$. The points
\begin{equation*}
\epsilon = (\epsilon_{1}, \cdots, \epsilon_{N_{i}}) \in Z^{i}
\end{equation*}
exactly correspond to the augmentations of $\Algebra^{i}$ by $\epsilon v_{ev, j} = \epsilon_{j}$. The following conditions are then equivalent:
\begin{enumerate}
\item $1 \in \mathcal{I}^{i}$.
\item $Z^{i} = \emptyset$.
\item $\Algebra^{i}$ does not admit an augmentation.
\end{enumerate}
The first equivalence follows from the weak Nullstellensatz and the second follows from the definition of an augmentation of $\Algebra^{i}$.

Suppose that $\Algebra^{i}$ does not admit an augmentation, so that $1 \in \mathcal{I}^{i}$. It follows that there are some $p_{j} \in \C[V^{i}_{ev}]$ such that
\begin{equation*}
1 = \sum p_{j}P(\partial v_{odd, j})
\end{equation*}
where the sum is taken over the $v_{odd, j} \in V^{i}_{odd}$. Since the $p_{j}$ are polynomials in the additive generators $v_{ev}$ of $V^{i}_{ev}$, each non-zero summand of $\partial p_{j}$ contains at least one generator $v_{odd} \in V^{i}_{odd}$ as a factor. Hence using the definition of $P$ and the Leibniz rule we obtain
\begin{equation*}
\partial(p_{j}v_{odd, j}) = (\partial p_{j})v_{odd, j} + p_{j}\partial v_{odd, j} \implies 1 = P(\partial x), \quad x = \sum p_{j} v_{odd,j}.
\end{equation*}

By the definition of $P$, we can write
\begin{equation*}
\partial x = 1 + z, \quad P(z) = 0
\end{equation*}
with $z$ being an even-degree element of $\Algebra^{i}$. Each summand of $z$ must have at least two $V^{i}_{odd}$ factors by degree considerations. Since squares of odd-degree generators vanish in $\Algebra$, there is some $m > 0$ for which $z^{m} = 0$. Moreover, for all $j > 0$
\begin{equation*}
\partial z^{j} =  j z^{j-1}\partial z = jz^{j-1}\partial^{2}x = 0.
\end{equation*}
It follows that
\begin{equation*}
\partial \sum_{j=1}^{m} (-1)^{j+1}z^{j-1} x = (1 + z) - z(1 + z) + \cdots + (-1)^{m-1}z^{m-1}(1 + z) = 1.
\end{equation*}
So we've established that if $\Algebra^{i}$ does not admit an augmentation for some $i > 0$ then the unit is exact in $\Algebra^{i}$ and hence $H\Algebra^{i} = 0$.

Now suppose that $H\Algebra \neq 0$. For each $i$, $H\Algebra^{i} \neq 0$ and hence $Z^{i} \neq \emptyset$ by our already-established results. Restrictions of augmentations define morphisms
\begin{equation*}
\pi_{i}: Z^{i+1} \to Z^{i},
\end{equation*}
given by projecting $V^{i+1}_{ev}$ to the subspace $V^{i}_{ev} \subset V^{i+1}_{ev}$. Augmentations of $\Algebra$ are exactly the sequences $(z_{i})_{i \geq 0}$ with $z_{i} \in Z^{i}$ and $\pi_{i}(z_{i+1}) = z_{i}$ for all $i \in \Z_{\geq 0}$.

Now we establish that
\begin{equation*}
Z^{\infty, i} = \cap_{I \geq i} Z^{I, i}
\end{equation*}
is non-empty for all $i$, where
\begin{equation*}
Z^{I, i} = \pi^{I,i}Z^{I} \subset \C^{N_{i}}, \quad \pi^{I, i} = \pi_{i}\cdots \pi_{I-1}.
\end{equation*}
By Chevalley's theorem (see \cite[Exercise II.3.19]{Hartshorne}), the $Z^{I,i}$ are constructible in the Zariski topology. They are each non-empty by assumption and nested, $Z^{I+1, i} \subset Z^{I,i}$.

Applying \Cref{Lemma:ConstructibleIntersection} to the $Z^{I,i}$, each $Z^{\infty,i}$ is non-empty. Next we show that
\begin{equation*}
\pi_{i}Z^{\infty, i+1} = Z^{\infty, i}.
\end{equation*}
The inclusion $\subset$ is clear. For the reverse inclusion, let $x \in Z^{\infty, i}$. For each $I > i$, the set
\begin{equation*}
F^x_{I} = \pi_{i}^{-1}(x) \cap Z^{I, i+1}
\end{equation*}
is constructible and non-empty, since $x \in Z^{I,i} = \pi_{i}Z^{I,i+1}$. The $F^x_{I}$ are also nested. By \Cref{Lemma:ConstructibleIntersection}, there is some $y \in \cap_{K} F^x_{K} \subset Z^{\infty, i+1}$, and $\pi_{i}(y) = x$.

Finally, choose any $z_{0} \in Z^{\infty, 0}$ and inductively choose $z_{i+1} \in Z^{\infty, i+1}$ with $\pi_{i}(z_{i+1}) = z_{i}$. The resulting sequence $(z_{i})$ defines an augmentation of $\Algebra$.
\end{proof}

In view of \Cref{rmk:finite}, for the purpose of this paper, the analogous statement for the action-filtered version of contact homology $\CH^{<L}(Y,\xi)$, generated by Reeb orbits of action $<L$, is sufficient.

\begin{remark}
It is important in \Cref{Lemma:aug} that the statement is only
\(\Z/2\Z\)-graded. The analogous statement with \(\Z\)-gradings does not hold in
general. For example, consider a DGA with generators \(x,y,z\) of degrees
\[
|x|=1,\qquad |y|=2,\qquad |z|=-2,
\]
and differential
\[
\partial x=1+yz,\qquad \partial y=\partial z=0.
\]
Then the homology is non-zero, but there is no \(\Z\)-graded augmentation.
\end{remark}
\section{Augmented symplectic cohomology and wrapped Floer cohomology}\label{s4}

To prove \Cref{thm:vanishingCH} we will follow the proof strategy of \Cref{cor:nofilling}, but using augmentations in place of Liouville fillings. The replacement of fillings with augmentations requires us to generalize the invariants described in \S \ref{SS3.1} and \S \ref{SS3.2}. We describe such generalizations in this section. The ideas presented here originate in work of Bourgeois--Oancea \cite{BO}.

Throughout this section, let \((W,\lambda)\) be a $2N$-dimensional Liouville cobordism from \(Y_-\) to \(Y_+\). We denote the completion of the cobordism by \(\widehat W\). Let \(L\subset W\) be an exact Lagrangian with connected positive Legendrian boundary \(\Lambda\subset Y_+\) and empty negative boundary. We assume that \(L\) is equipped with a relatively spin structure (which will allow us to coherently orient moduli spaces of holomorphic curves).

We assume that the contact forms \(\lambda|_{Y_\pm}\) are nondegenerate and that the Reeb chords of \(\Lambda\) are nondegenerate, at least up to the action thresholds relevant in each construction below. Fix an augmentation \(\epsilon\) of the contact homology algebra of \((Y_-,\lambda|_{Y_-})\), after choosing the auxiliary data required by the virtual perturbation setup. All moduli spaces below are defined using auxiliary data compatible with those defining \(\epsilon\); we suppress this dependence from the notation. Since we seek to leverage \Cref{Lemma:aug} in applications, we assume throughout that all (co)homology groups use $\C$ coefficients and $\Z/2\Z$ gradings.

\subsection{Augmented wrapped Floer cohomology \(HW^*(L,\epsilon)\)}

Fix a marked point on the image of every embedded Reeb orbit of \(Y_-\). Let
\([\Gamma]\) be a multiset of good Reeb orbits of \(Y_-\), and let
\[
\Gamma=\{\gamma_i\}_{1\le i\le m}
\]
be an ordered representative of \([\Gamma]\). Let \(H\) be a Hamiltonian as in
the definition of wrapped Floer cohomology, chosen to vanish on \(W\). Let
\(x,y\) be non-constant Hamiltonian chords in the positive end of \(W\). We
define the compactified moduli space
\begin{equation}\label{eqn:wrapped_aug}
\cM_{x,y,\Gamma}
=
\overline{
\left\{
\begin{array}{c}
u:\R_s\times [0,1]_t\setminus\{\zeta_1,\ldots,\zeta_m\}\to \widehat{W}\\
\text{asymptotic markers at } \zeta_i
\end{array}
\left|
\begin{array}{c}
\partial_su+J(\partial_tu-X_H)=0,\\
u(\cdot,0),u(\cdot,1)\in \widehat{L},\\
\displaystyle\lim_{s\to\infty}u(s,\cdot)=x,\\
\displaystyle\lim_{s\to-\infty}u(s,\cdot)=y,\\
u(\zeta_i)=\gamma_i
\end{array}
\right.
\right\}/\R}.
\end{equation}
Here the \(\zeta_i\) are distinct free interior negative punctures, and
\(u(\zeta_i)=\gamma_i\) means that \(u\) converges to
\((-\infty,\gamma_i)\) in the SFT sense, with the asymptotic marker mapping to
the fixed marked point on \(\Ima\gamma_i\). The compactification allows
Floer-type breaking at non-constant Hamiltonian chords and SFT breaking (contact homology breaking) at the
negative punctures. More precisely, as we fix the conformal structure on $\R\times [0,1]$, in addition to Floer breakings, we have (1) collision of marked points leading to a degeneration involving a rooted tree of holomorphic rational curves in $\widehat{Y}_-$, (2) the separation of marked points to infinity leading to a breaking at a Hamiltonian chord in the middle, (3) the development of holomorphic planes in $\widehat{Y}_-$, i.e.\ ``sphere bubbles'' in this context.

Let \((h,g)\) be a Morse--Smale pair on \(L\) such that the gradient of \(h\)
points outward along \(\partial L\). Let \(p\) be a critical point of \(h\), let
\(x\) be a non-constant Hamiltonian chord, and let \(\Gamma\) be an ordered
multiset of good Reeb orbits on \(Y_-\). We define
\[
\cM_{x,p,\Gamma}
=
\overline{
\left\{
\begin{array}{c}
u:\R_s\times [0,1]_t\setminus\{\zeta_1,\ldots,\zeta_m\}\to \widehat{W}\\
\text{asymptotic markers at } \zeta_i\\
\gamma:(-\infty,0]_s\to L
\end{array}
\left|
\begin{array}{c}
\partial_su+J(\partial_tu-X_H)=0,\\
\gamma'(s)+\nabla_g h=0,\\
u(\cdot,0),u(\cdot,1)\in \widehat{L},\\
\displaystyle\lim_{s\to\infty}u(s,\cdot)=x,\\
\displaystyle\lim_{s\to-\infty}u(s,\cdot)=\gamma(0)\in L^\circ\subset W^\circ,\\
\displaystyle\lim_{s\to-\infty}\gamma(s)=p,\\
u(\zeta_i)=\gamma_i
\end{array}
\right.
\right\}/\R}.
\]
Since the stable manifold of \(p\) is precompact in the interior of \(L\), the
negative end of \(u\) lands in the interior of \(W\), where the Floer equation
becomes the Cauchy--Riemann equation with totally real boundary condition.
Thus the negative boundary puncture is removable. The compactification allows
breaking at non-constant chords on the positive end, breaking of the Morse
trajectory, and SFT breaking at the negative punctures, as before.

Let \(C(L,H,\epsilon)\) be the \(\bm{k}\)-vector space generated by the
non-constant Hamiltonian chords \(x\) and the critical points \(p\) of \(h\). The
differential has three components:
\begin{eqnarray*}
\delta_0(p) &=& \text{the Morse differential};\\
\delta_+(x)
&=&
\sum_y\sum_{[\Gamma]}
\frac{1}{\mu_\Gamma\kappa_\Gamma}
\#\cM_{x,y,\Gamma}\,\epsilon(q^\Gamma)\,y;\\
\delta_{+,0}(x)
&=&
\sum_p\sum_{[\Gamma]}
\frac{1}{\mu_\Gamma\kappa_\Gamma}
\#\cM_{x,p,\Gamma}\,\epsilon(q^\Gamma)\,p.
\end{eqnarray*}
The sums are over zero-dimensional moduli spaces, counted using the chosen
virtual perturbation scheme. Here \(q^\Gamma=q_{\gamma_1}\cdots q_{\gamma_m}\).

The subspace \(C_0(L,H,\epsilon)\) generated by the critical points of \(h\) is
the Morse cochain complex of \(h\). The quotient complex
\(C_+(L,H,\epsilon)\) is generated by the non-constant Hamiltonian chords and has
differential \(\delta_+\). The total complex \(C(L,H,\epsilon)\) is the mapping
cone of \(\delta_{+,0}\).

\begin{proposition}\label{prop:delta^2}
The differential satisfies \(\delta^2=0\).
\end{proposition}

\begin{proof}
We give the standard boundary-of-one-dimensional-moduli-spaces argument,
suppressing orientation signs and the details of the virtual perturbation
scheme; see \Cref{rmk:ana} below. It suffices to show
\[
\delta_+^2=0
\qquad\text{and}\qquad
\delta_{+,0}\circ\delta_+ + \delta_0\circ\delta_{+,0}=0.
\]

For \(\delta_+^2=0\), consider a compactified moduli space
\(\cM_{x,y,\Gamma}\) of virtual dimension \(1\). Its codimension-one boundary
has two types of contributions:
\begin{enumerate}
\item Floer breaking:
\[
\cM_{x,z,\Gamma_1}\times \cM_{z,y,\Gamma_2},
\qquad
\Gamma=\Gamma_1\sqcup \Gamma_2.
\]
\item SFT breaking at a negative puncture:
\[
\bigl(\cM_{x,y,\Gamma_1\sqcup\{\gamma\}}\times
\cM_{\gamma,\Gamma_2}\bigr)/(\Z/\kappa_\gamma),
\qquad
\Gamma=\Gamma_1\sqcup \Gamma_2.
\]
This comes from the previously mentioned ``bubbling'' of holomorphic planes in $\widehat{Y}_-$, as well as the collision of negative punctures.
\end{enumerate}
A priori, one might also have a boundary stratum in which an interior negative
puncture approaches the boundary, producing one of more disk bubbles
\[
\nu:(\D\setminus\{0\},\partial\D)\to(\widehat W,\widehat L)
\]
with an interior negative puncture asymptotic to a Reeb orbit \(\gamma\) on
\(Y_-\). Such bubbles do not occur. Indeed, write
\[
W_+=W\cup (1,\infty)\times Y_+,\qquad
W_-=(0,1)\times Y_-,
\]
and let \(\lambda_-=\lambda|_{Y_-}\). Then
\begin{equation}\label{eqn:diskbubble}
\int_{\nu^{-1}(W_+)}\nu^*\rd\lambda
+
\int_{\nu^{-1}(W_-)}\nu^*\rd\lambda_-
\end{equation}
is non-negative because \(\nu\) is \(J\)-holomorphic and \(J\) is compatible with
the relevant symplectic forms/contact forms. On the other hand, by Stokes'
theorem and the exactness of \(\widehat L\), the same quantity equals
\[
-\int_\gamma \lambda_-<0,
\]
a contradiction.

After weighting all negative punctures by the augmentation \(\epsilon\), the
SFT-breaking contributions cancel precisely by the augmentation equation for
\(\epsilon\), while the Floer-breaking contributions give \(\delta_+^2=0\). The
identity
\(\delta_{+,0}\circ\delta_+ + \delta_0\circ\delta_{+,0}=0\) is proved in the
same way, using the compactified one-dimensional moduli spaces
\(\cM_{x,p,\Gamma}\).
\end{proof}

The filtered versions are obtained by restricting to critical points and/or
non-constant Hamiltonian chords corresponding to Reeb chords of action less than
a fixed threshold \(a\). We obtain cohomology groups
\[
HW^*(L,\epsilon),\quad
HW^*_+(L,\epsilon),\quad
HW^{*,<a}(L,\epsilon),\quad
HW^{*,<a}_+(L,\epsilon),
\]
which fit into long exact sequences
\begin{equation}\label{eq:AugmentedWFLES}
\cdots \to H^*(L)
\to HW^{*,<a}(L,\epsilon)
\to HW^{*,<a}_+(L,\epsilon)
\to H^{*+1}(L)
\to \cdots,
\end{equation}
where \(0\le a\le \infty\).

One can also define these invariants using a Hamiltonian \(H\) that is
\(C^2\)-small on \(W\), vanishes near \(Y_-\) so that it extends by zero to the
negative cylindrical end, and has nondegenerate constant chords on \(L\).

\begin{remark}\label{rmk:ana}
We record several comments on the analytic foundations and invariance properties
of \(HW^*(L,\epsilon)\).
\begin{enumerate}
\item The open strata of \(\cM_{x,y,\Gamma}\) and \(\cM_{x,p,\Gamma}\),
consisting of smooth curves, can be arranged to be cut out transversely for
a generic almost complex structure, using the Hamiltonian perturbation and
somewhere-injectivity. However, the contact homology moduli spaces
\(\cM_{\gamma,\Gamma}\) on \(Y_-\) are not generally transversely cut out.
Therefore, one must use a virtual perturbation theory, such as Pardon's
virtual fundamental cycles \cite{Pardon:Contact} or the semi-global
Kuranishi charts of Bao--Honda \cite{Bao-Honda}, both for the contact
homology moduli spaces and for the augmented wrapped Floer moduli spaces.

\item Fix a contact form \(\alpha_-=\lambda|_{Y_-}\), an
\(\alpha_-\)-compatible almost complex structure \(J\) on the negative end,
and auxiliary data \(\theta\in\Theta_I(Y_-,\alpha_-,J)\). After extending
these data, one obtains \(HW^*(L,\epsilon)\). A priori, this cohomology
depends on the chosen extension data. Standard continuation maps and
homotopies of continuation maps, as in \cite[\S 10]{Pardon}, show that it
is independent of the extension data once \(\alpha_-\), \(J\), \(\theta\),
and the augmentation \(\epsilon\) are fixed. Removing dependence on
\(\alpha_-\), \(J\), and \(\theta\) requires a chain-level homotopy theory
for the contact homology DGA; this will be developed in \cite{DG}. For the
purposes of this paper, it is enough to work with \(HW^*(L,\epsilon)\)
relative to the fixed data defining \(\epsilon\).

\item Although the chain-level well-definedness and functoriality for contact homology are not established yet, the set of homotopy classes of augmentations is indeed a contact invariant by the work of Chaidez \cite{chaidez2024contact}. Hence the set of linearized contact homologies is a contact invariant. Consequently, the set of augmented wrapped Floer cohomologies associated to a Lagrangian $L$ (with augmentations varying) is an invariant following the arguments of \cite{chaidez2024contact}.

\item For computations, the virtual perturbation theory should satisfy the following compatibility property: whenever the compactified moduli space is transversely cut out, the virtual count agrees with the geometric signed count. In Pardon's framework for contact homology, this is \cite[Proposition 4.33]{Pardon:Contact}.

\item The same comments apply to all augmented constructions below.
\end{enumerate}
\end{remark}

The augmented wrapped Floer cohomology \(HW^*(L,\epsilon)\) carries a product
structure analogous to the usual wrapped product. Using a \(C^2\)-small
Hamiltonian on \(W\), consider the compactified moduli spaces
\begin{equation}
\cM_{x,y,z,\Gamma}
=
\overline{
\left\{
\begin{array}{c}
u:(\D_{2,1}\setminus\{\zeta_1,\ldots,\zeta_m\},\partial\D_{2,1})
\to(\widehat W,\widehat L)\\
\zeta_i\in \operatorname{Int}(\D_{2,1})\\
\text{asymptotic markers at } \zeta_i
\end{array}
\left|
\begin{array}{c}
(\rd u-X_H\otimes\beta)^{0,1}=0,\\
u(\partial\D_{2,1})\subset\widehat L,\\
u(z_{+,1})=x,\quad u(z_{+,2})=y,\\
u(z_-)=z,\quad u(\zeta_i)=\gamma_i
\end{array}
\right.
\right\}},
\label{eqn:mu_2_aug}
\end{equation}
where \(x,y,z\) are Hamiltonian chords and \(\Gamma=\{\gamma_i\}\) is an ordered
multiset of good Reeb orbits on \(Y_-\). Define
\[
\mu_2(x,y)
=
\sum_{z,[\Gamma]}
\frac{1}{\mu_\Gamma\kappa_\Gamma}
\#\cM_{x,y,z,\Gamma}\,\epsilon(q^\Gamma)\,z.
\]

\begin{proposition}\label{prop:product}
For \(a,b>0\), the product induces
\[
\mu_2:
HW^{*,<a}(L,\epsilon)\otimes HW^{*,<b}(L,\epsilon)
\to
HW^{*,<a+b}(L,\epsilon).
\]
For \(a,b\ll 1\) sufficiently small, this product agrees with the cup product structure on \(H^*(L)\). The image of $1 \in H^{0}(L)$ under the inclusion $H^{\ast}(L) \to HW^{*,<a}(L,\epsilon)$ is a two-sided unit, by the compatibility of $\mu_2$ with continuation maps.
\end{proposition}

\begin{proof}
The proof follows the construction of TQFT operations in \cite{Ritter1}, except that curves are now allowed to have interior negative punctures asymptotic to Reeb orbits on \(Y_-\). These punctures are algebraically weighted by the augmentation \(\epsilon\). As in the proof of \Cref{prop:delta^2}, interior negative punctures cannot escape to the boundary, by the exactness/energy argument above. The codimension-one SFT-breaking contributions cancel by the augmentation equation, and the remaining boundary strata give the \(A_\infty\)-type relation showing that \(\mu_2\) is a chain map.

For the last claim, consider the energy of a curve \(u\) contributing to \eqref{eqn:mu_2_aug}:
\[
\int_{u^{-1}(W_+)}
\frac{1}{2}|\rd u-X_H\otimes\beta|^2
+
\int_{u^{-1}(W_-)}u^*\rd\lambda_-,
\]
where \(W_+=W\cup(1,\infty)\times Y_+\), \(W_-=(0,1)\times Y_-\), and \(\lambda_-=\lambda|_{Y_-}\). This quantity is non-negative. By Stokes' theorem and the condition \(\rd\beta\le 0\), it is bounded above by
\[
\cA_{(a+b)H}(z)-\cA_{aH}(x)-\cA_{bH}(y)
-\sum_{\gamma_i\in\Gamma}\int_{\gamma_i}\lambda_-.
\]
If \(a,b\) are sufficiently small, then all relevant Hamiltonian chords are constant chords and their actions are smaller than the minimal period of a Reeb orbit on \(Y_-\). Therefore no negative punctures can occur. The operation then reduces to the usual product for exact Lagrangian fillings, which is the cup product on \(H^*(L)\).
\end{proof}

As in the proof of \Cref{prop:Euler}, we obtain the following Euler
characteristic computation.

\begin{proposition}\label{prop:Euler1}
In the setup of \Cref{prop:contact}, suppose that the positive boundary is
\((\bbS^{2n+1},\xi_{\std})\), that \(L\) has positive Legendrian boundary
\(\Lambda\), and that \(1\le a<9\pi-1\). Then
\[
\chi\bigl(HW^{*,<a}_+(L,\epsilon)\bigr)
=
(-1)^{\frac{n(n-1)}{2}}\tb(\Lambda).
\]
\end{proposition}

\begin{proof}
The filtered complex \(C_+^{<a}(L,H,\epsilon)\) is generated by the non-constant
Hamiltonian chords corresponding to Reeb chords of \(\Lambda\) of action less
than \(a\). Since \(1\le a<9\pi-1\), the action gap in \Cref{prop:contact}
implies that these are precisely the Reeb chords of action less than \(1\).
Thus the Euler characteristic of the cohomology equals the Euler characteristic
of this finite-dimensional chain complex. The same grading computation as in
\Cref{prop:Euler}, using \cite[Proposition 3.3]{EES}, gives
\[
\chi\bigl(HW^{*,<a}_+(L,\epsilon)\bigr)
=
(-1)^{\frac{n(n-1)}{2}}\tb(\Lambda).
\]
\end{proof}

\subsection{Augmented symplectic cohomology of the cobordism \(SH^*(W,\epsilon)\)}

Let \((f,g)\) be a Morse--Smale pair on \(W\) such that the gradient of \(f\)
points outward along both boundary components of \(W\). Let \(p\) be a critical
point of \(f\), let \(x,y\) be non-constant Hamiltonian orbits, and let
\(\Gamma\) be an ordered multiset of good Reeb orbits on \(Y_-\). Define
\begin{equation}\label{eqn:partial}
\cM_{x,y,\Gamma}
=
\overline{
\left\{
\begin{array}{c}
u:\R_s\times\bbS^1_t\setminus\{\zeta_1,\ldots,\zeta_m\}\to\widehat W\\
\text{asymptotic markers at } \zeta_i
\end{array}
\left|
\begin{array}{c}
\partial_su+J(\partial_tu-X_H)=0,\\
\displaystyle\lim_{s\to\infty}u(s,\cdot)=x,\\
\displaystyle\lim_{s\to-\infty}u(s,\cdot)=y,\\
u(\zeta_i)=\gamma_i
\end{array}
\right.
\right\}/\R},
\end{equation}
and
\[
\cM_{x,p,\Gamma}
=
\overline{
\left\{
\begin{array}{c}
u:\R_s\times\bbS^1_t\setminus\{\zeta_1,\ldots,\zeta_m\}\to\widehat W\\
\text{asymptotic markers at } \zeta_i\\
\gamma:(-\infty,0]\to W
\end{array}
\left|
\begin{array}{c}
\partial_su+J(\partial_tu-X_H)=0,\\
\gamma'(s)+\nabla f=0,\\
\displaystyle\lim_{s\to\infty}u(s,\cdot)=x,\\
\displaystyle\lim_{s\to-\infty}u(s,\cdot)=\gamma(0)\in W^\circ,\\
\displaystyle\lim_{s\to-\infty}\gamma(s)=p,\\
u(\zeta_i)=\gamma_i
\end{array}
\right.
\right\}/\R}.
\]
Since the stable manifold of \(p\) is precompact and contained in \(W^\circ\),
the negative end of \(u\) lands in the interior of \(W\), where the Floer
equation becomes the Cauchy--Riemann equation with a removable singularity.

Let \(C(H,\epsilon)\) be the cochain complex generated by critical points of
\(f\) and non-constant Hamiltonian orbits. Define its differential by
\begin{eqnarray*}
\delta_0(p) &=& \text{the Morse differential};\\
\delta_+(x)
&=&
\sum_y\sum_{[\Gamma]}
\frac{1}{\mu_\Gamma\kappa_\Gamma}
\#\cM_{x,y,\Gamma}\,\epsilon(q^\Gamma)\,y;\\
\delta_{+,0}(x)
&=&
\sum_p\sum_{[\Gamma]}
\frac{1}{\mu_\Gamma\kappa_\Gamma}
\#\cM_{x,p,\Gamma}\,\epsilon(q^\Gamma)\,p.
\end{eqnarray*}
Then \(C_0(H,\epsilon)\), generated by critical points, is the Morse cochain
complex of \(W\), while \(C_+(H,\epsilon)\), generated by non-constant orbits, is
the quotient complex with differential \(\delta_+\). The same argument as in
\Cref{prop:delta^2} shows that \(\delta^2=0\).\footnote{Floer breaking cannot occur at Reeb orbits in the negative end of the cobordism, only at non-constant Hamiltonian orbits. This is because the lower component of a configuration of broken curves will have non-positive energy coming from the symplectic action, since $H=0$ on the negative end. The SFT breaking is as before.} Hence we obtain augmented symplectic cohomology groups
\[
SH^{*,<a}(W,\epsilon),
\qquad
SH^{*,<a}_+(W,\epsilon),
\]
fitting into the long exact sequence
\begin{equation}\label{eq:AugmentedSHLES}
\cdots \to H^*(W)
\to SH^{*,<a}(W,\epsilon)
\to SH^{*,<a}_+(W,\epsilon)
\stackrel{\delta_{+,0}}{\longrightarrow}
H^{*+1}(W)
\to \cdots .
\end{equation}

\begin{remark}\label{rmk:ring}
The group \(SH^*(W,\epsilon)\) is not a ring. It should be viewed as one part
of the symplectic cohomology of a cobordism in the sense of
Cieliebak--Oancea \cite{CieOan}.\footnote{Cieliebak--Oancea use symplectic
homology conventions. With our grading convention, their symplectic homology is
identified with \(SH^{N-*}\), without dualizing.} The missing piece is dual to
the positive symplectic cohomology of the negative boundary. The full
cobordism theory has a ring structure and generalizes Rabinowitz Floer
cohomology, which corresponds to the trivial cobordism.
\end{remark}

We will use the composition
\[
\delta_{\partial_+}:
SH_+^{*,<a}(W,\epsilon)
\longrightarrow H^{*+1}(W)
\longrightarrow H^{*+1}(Y_+).
\]
This map has the following direct description. Fix a Morse function \(h_+\) on
\(Y_+\), together with a generic metric. For a Hamiltonian orbit \(x\) and a
critical point \(p\) of \(h_+\), define
\begin{equation}\label{eqn:partial_+}
\cM^{\partial_+}_{x,p,\Gamma}
=
\overline{
\left\{
\begin{array}{c}
u:\R_s\times\bbS^1_t\setminus\{\zeta_1,\ldots,\zeta_m\}\to\widehat W\\
\text{asymptotic markers at } \zeta_i\\
\gamma:(-\infty,0]\to Y_+
\end{array}
\left|
\begin{array}{c}
\partial_su+J(\partial_tu-X_H)=0,\\
\gamma'(s)+\nabla h_+=0,\\
\displaystyle\lim_{s\to\infty}u(s,\cdot)=x,\\
\displaystyle\lim_{s\to-\infty}u(s,\cdot)=\gamma(0)\in Y_+\subset\widehat W,\\
\displaystyle\lim_{s\to-\infty}\gamma(s)=p,\\
u(\zeta_i)=\gamma_i
\end{array}
\right.
\right\}/\R}.
\end{equation}
Then the coefficient of \(p\) in \(\delta_{\partial_+}(x)\) is
\[
\sum_{[\Gamma]}
\frac{1}{\mu_\Gamma\kappa_\Gamma}
\epsilon(q^\Gamma)\,
\#\cM^{\partial_+}_{x,p,\Gamma}.
\]
This description follows by the same argument as
\cite[Proposition 3.2]{Zhou:ADC}.

\subsection{The \(\CO\) map for augmented symplectic cohomology}

We now adapt the closed-open map from \S \ref{ss:3.3} to Liouville cobordisms with
augmentations of the chain-level contact homology algebras of their negative ends. Fix a Morse function \(f\) on \(W\) such that the gradient of \(f\) is transverse to both boundary components and points outward along them. In a Weinstein neighborhood of \(L\), assume
\[
f=\pi^*h+\varepsilon\rho^2,
\]
where \(\varepsilon>0\), \(\pi:T^*L\to L\), and \(\rho\) is the length function
associated to a bundle metric on \(T^*L\). With a suitable metric on \(W\) that
has the corresponding splitting near \(L\), the critical points on \(L\) define
a quotient complex of the Morse cochain complex on \(W\), canonically
identified with the Morse cochain complex of \(h\) on \(L\). This quotient map
induces the restriction map \(H^*(W)\to H^*(L)\).

Define the compactified moduli spaces
\[
\cM^{\CO}_{x,y,\Gamma}
:=
\overline{
\left\{
\begin{array}{c}
u:(\D_{0,1|1,0}\setminus\{\zeta_1,\ldots,\zeta_m\},
\partial\D_{0,1|1,0})
\to(\widehat W,\widehat L)\\
\text{asymptotic markers at } \zeta_i
\end{array}
\left|
\begin{array}{c}
(\rd u-X_H\otimes\beta)^{0,1}=0,\\
u(\partial\D_{0,1|1,0})\subset\widehat L,\\
\displaystyle\lim_{s\to\infty}u=x,\quad
\displaystyle\lim_{s\to-\infty}u=y,\\
u(\zeta_i)=\gamma_i
\end{array}
\right.
\right\}/\R},
\]
where \(x\) is a non-constant Hamiltonian orbit of \(aH\), \(y\) is a
non-constant Hamiltonian chord of \(bH\), \(s\) is the cylindrical coordinate
near the punctures, and the \(\zeta_i\) are interior negative punctures. Also
define
\[
\cM^{\CO}_{x,q,\Gamma}
:=
\overline{
\left\{
\begin{array}{c}
u:(\D_{0,1|1,0}\setminus\{\zeta_1,\ldots,\zeta_m\},
\partial\D_{0,1|1,0})
\to(\widehat W,\widehat L)\\
\text{asymptotic markers at } \zeta_i\\
\gamma:(-\infty,0]\to L
\end{array}
\left|
\begin{array}{l}
(\rd u-X_H\otimes\beta)^{0,1}=0,\\
\partial_s\gamma+\nabla h=0,\\
u(\partial\D_{0,1|1,0})\subset\widehat L,\\
\displaystyle\lim_{s\to-\infty}\gamma=q,\\
\displaystyle\lim_{s\to-\infty}u=\gamma(0),\\
\displaystyle\lim_{s\to\infty}u=x,\\
u(\zeta_i)=\gamma_i
\end{array}
\right.
\right\}/\R},
\]
where \(x\) is a non-constant Hamiltonian orbit of \(aH\), and \(q\) is a
critical point of \(h\). Since the stable manifold of \(q\) is precompact in the
interior of \(L\), where the Hamiltonian vanishes, the negative boundary
puncture of \(u\) is removable, and the condition
\(\lim_{s\to-\infty}u=\gamma(0)\) is well-defined.

The map
\[
\CO:C(H,\epsilon)\to C(L,H,\epsilon)
\]
is the sum of three components. The first,

\[
\CO_+:C_+(H,\epsilon)\to C_+(L,H,\epsilon),
\]
is defined by
\[
\CO_+(x)
=
\sum_y\sum_{[\Gamma]}
\frac{1}{\mu_\Gamma\kappa_\Gamma}
\epsilon(q^\Gamma)\,
\#\cM^{\CO}_{x,y,\Gamma}\,y.
\]
The second,
\[
\CO_0:C_0(H,\epsilon)\to C_0(L,H,\epsilon),
\]
is the quotient map from the Morse cochain complex of \(f\) to that of \(h\).
The third,
\[
\CO_{+,0}:C_+(H,\epsilon)\to C_0(L,H,\epsilon),
\]
is defined by
\[
\CO_{+,0}(x)
=
\sum_q\sum_{[\Gamma]}
\frac{1}{\mu_\Gamma\kappa_\Gamma}
\epsilon(q^\Gamma)\,
\#\cM^{\CO}_{x,q,\Gamma}\,q.
\]

As in the unaugmented case, these maps give the following compatibility between
restriction and the closed-open map.

\begin{proposition}\label{prop:CO}
For \(a>0\), the following diagram commutes on cohomology:
\[
\xymatrix{
H^*(W) \ar[r]\ar[d] &
H^*(L)\ar[d]\\
SH^{*,<a}(W,\epsilon) \ar[r]^{\CO} &
HW^{*,<a}(L,\epsilon).
}
\]
\end{proposition}

The proof here is the same as that of Equation \eqref{eqn:CO}.

\section{Proof of the main results}\label{s5}

Having established the augmented versions of wrapped Floer and symplectic cohomologies in \S \ref{s4}, we are now ready to establish the proofs of our main results.

\subsection{Proof of \Cref{thm:vanishingCH}}

We first record the algebraic analogue of the key property of the contact form $\alpha$ constructed in \Cref{prop:contact}. Choose \(0<\delta\ll 1\) so that
\[
4\pi+\delta<9\pi-1.
\]
For a Liouville filling \(W\) of \((\bbS^{2n+1},\xi_{\std})\), the proof of \cite[Theorem 1.1]{product}, or equivalently \cite[Proposition 2.11]{Zhou:RP2n-1}, shows that the image of the composition
\begin{equation}\label{eq:SHCompositionCoversUnit}
SH^{*,<4\pi+\delta}_+(W)
\longrightarrow H^{*+1}(W)
\longrightarrow H^{*+1}(\bbS^{2n+1})
\end{equation}
contains the unit \(1\in H^0(\bbS^{2n+1})\). This is one of the ingredients in the symplectic cohomology proof of the Eliashberg--Floer--McDuff theorem \cite{McDuff} appearing in \cite[Theorem 1.2]{product}.

The neck-stretching arguments of \cite{product, Zhou:RP2n-1} apply when the filling is replaced by an exact cobordism together with an augmentation of the chain-level contact homology algebra of the negative end.

\begin{proposition}\label{prop:sphere}
Let \(W\) be a Liouville cobordism from a contact manifold \(Y_-\) to
\((\bbS^{2n+1},\xi_{\std})\). Let \(\epsilon\) be an augmentation of the contact homology algebra of \(Y_-\), defined by fixing geometric and perturbative data as in \cite{Bao-Honda, Pardon:Contact}. Assume that the Liouville form on \(W\), restricted to the positive boundary \(\bbS^{2n+1}\) of the cobordism, is the contact form $\alpha$ of \Cref{prop:contact}. Then the map
\[
\delta_{\partial_+}:
SH^{*,<4\pi+\delta}_+(W,\epsilon)
\longrightarrow H^{*+1}(\bbS^{2n+1})
\]
has image containing \(1\in H^0(\bbS^{2n+1})\).
\end{proposition}

\begin{proof}
The proof follows the strategies of \cite[Theorem 1.1]{product} and \cite[Proposition 2.11]{Zhou:RP2n-1}, with the negative ends of holomorphic curves algebraically capped off by the augmentation \(\epsilon\).

More concretely, consider the generator \(\check{\gamma}_{\min}\) associated to the shortest Reeb orbit \(\gamma_{\min}\) of the contact form $\alpha$, following the notation of \cite{product}. The holomorphic curves whose counts define the map \(\delta_{\partial_+}\) are in the moduli spaces \(\cM^{\partial_+}_{\check{\gamma}_{\min},p,\Gamma}\), where \(p\) represents the unit in \(H^0(\bbS^{2n+1})\). The relevant curves for the differential of \(\check{\gamma}_{\min}\) lie in the moduli spaces \(\cM_{\check{\gamma}_{\min},y,\Gamma}\).

We shall apply a neck-stretching along the positive boundary of $W$, i.e. $(\bbS^{2n+1},\xi_{\std})$. By the same action and index estimates used in \cite[\S 4.2.2]{BGMZ}, holomorphic curves appearing in the SFT buildings associated to such a degeneration cannot be negatively asymptotic to Reeb orbits in $(\bbS^{2n+1},\xi_{\std})$. As usual, we exclude trivial cylinders from such buildings. In particular, \(\cM_{\check{\gamma}_{\min},y,\Gamma}\) and \(\cM^{\partial_+}_{\check{\gamma}_{\min},p,\Gamma}\) are empty whenever $\Gamma \ne \emptyset$. Elements of these $\Gamma =\emptyset$ moduli spaces can be identified with the corresponding curves in the Liouville filling obtained by rounding the corners of
\[
\B^2(3)\times \B^{2n-2}(3)\times \B^2(2),
\]
described in the construction of $\alpha$. Using this filling, the
map $\delta_{\partial_{+}}$ hits \(1\in H^0(\bbS^{2n+1})\) by \cite[Proposition 3.2 and proof of Theorem 1.1]{product}. Hence the same is true for \(SH^{*,<4\pi+\delta}_+(W,\epsilon)\).
\end{proof}

Let \(W_\Lambda\) be the Liouville cobordism from
\[
(\bbS^{2n+1}_{+1}(\Lambda),\xi_{+1}(\Lambda))
\]
to \((\bbS^{2n+1},\xi_{\std})\) determined by contact $(+1)$-surgery along $\Lambda$. Choose the Liouville form so that its restriction to the positive boundary is the contact form $\alpha$ of \Cref{prop:contact}. Let \(L\subset W_\Lambda\) be the exact Lagrangian disk with positive Legendrian boundary \(\Lambda\) naturally determined by the surgery.

\begin{proof}[Proof of \Cref{thm:vanishingCH}]
Suppose, by contradiction, that
\((\bbS^{2n+1}_{+1}(\Lambda),\xi_{+1}(\Lambda))\) has non-zero contact homology.
By \Cref{Lemma:aug}, its chain-level contact homology CDGA admits a \(\Z/2\Z\)-graded, $\C$-valued augmentation, \(\epsilon\).

By \Cref{prop:sphere}, the map
\[
\delta_{\partial_+}:
SH^{*,<4\pi+\delta}_+(W_\Lambda,\epsilon)
\longrightarrow H^{*+1}(\bbS^{2n+1})
\]
has image containing \(1\in H^0(\bbS^{2n+1})\). Hence, by exactness of the sequence of Equation \eqref{eq:AugmentedSHLES} for \(SH^*(W_\Lambda,\epsilon)\), there exists a
class
\[
1+x \in H^{\mathrm{even}}(W_\Lambda;\C),
\qquad
x \in \bigoplus_{i>0}H^{2i}(W_\Lambda;\C),
\]
whose image in \(SH^{*,<4\pi+\delta}(W_\Lambda,\epsilon)\) vanishes.

Applying the closed-open map and using \Cref{prop:CO}, the restriction
\[
(1+x)|_L\in H^*(L;\C)
\]
maps to zero in \(HW^{*,<4\pi+\delta}(L,\epsilon)\). Since \(L\) is a disk,
\((1+x)|_L=1\in H^0(L;\C)\).

According to \Cref{prop:contact}, there are no Reeb chords of \(\Lambda\) with action in \([1,4\pi+\delta)\). Therefore
\[
HW^{*,<4\pi+\delta}(L,\epsilon)=HW^{*,<1}(L,\epsilon).
\]
Moreover, while the product a priori lands in \(HW^{*,<2}(L,\epsilon)\), we may identify \(HW^{*,<2}(L,\epsilon)\) with \(HW^{*,<1}(L,\epsilon)\) by the same action considerations. Indeed, constants $a=b=1$ appearing in \Cref{prop:product} are clearly applicable when using the contact form $\alpha$ of \Cref{prop:contact}. Thus \(HW^{*,<1}(L,\epsilon)\) inherits a unital ring structure. Since its unit is zero, we conclude that
\[
HW^{*,<1}(L,\epsilon)=0.
\]

The tautological exact sequence of Equation \eqref{eq:AugmentedWFLES} then gives
\[
\chi\bigl(HW^{*,<1}_+(L,\epsilon)\bigr)=-\chi(L)=-1.
\]
On the other hand, by \Cref{prop:Euler1},
\[
\chi\bigl(HW^{*,<1}_+(L,\epsilon)\bigr)
=
(-1)^{\frac{n(n-1)}{2}}\tb(\Lambda).
\]
Therefore
\[
\tb(\Lambda)=(-1)^{\frac{n(n-1)}{2}+1},
\]
contradicting the hypothesis. Hence the contact homology of
\((\bbS^{2n+1}_{+1}(\Lambda),\xi_{+1}(\Lambda))\) vanishes.
\end{proof}

\subsection{Proof of the three-dimensional case}

\begin{proof}[Proof of \Cref{cor:3D}]
For \(k=1\), algebraic overtwistedness of $(\bbS^{3}_{+1/k}(\Lambda),\xi_{+1/k}(\Lambda))$ follows immediately from \Cref{thm:vanishingCH}, since in dimension \(3\) the contact homology can only be non-zero when $\tb(\Lambda)=-1$.

For general \(k\ge 1\), use the standard Ding--Geiges description of contact
\((+1/k)\)-surgery as contact \((+1)\)-surgeries on $k$ copies of $\Lambda$, given by Legendrian push-offs \cite{DingGeiges:consur}. Thus there is a Liouville cobordism $W$ from $(\bbS^{3}_{+1/k}(\Lambda),\xi_{+1/k}(\Lambda))$ to $(\bbS^{3}_{+1}(\Lambda),\xi_{+1}(\Lambda))$. Since the positive boundary of $W$ has vanishing contact homology, the negative boundary of $W$ must have vanishing contact homology as well, by functoriality.

Tightness of the $(\bbS^{3}_{+1/k}(\Lambda),\xi_{+1/k}(\Lambda))$ follows from \cite{Golla,LS}: under the hypothesis
\[
\tb(\Lambda)=2g_s(\Lambda)-1,
\]
the contact Ozsv\'ath--Szab\'o invariant of \((\bbS^3_{+1/k}(\Lambda),\xi_{+1/k}(\Lambda))\) is non-zero for every positive integer \(k\), and the non-vanishing of this invariant implies tightness.

Finally, suppose that \(\Lambda\) bounds an exact, orientable Lagrangian surface \(L\) that
is not a disk. By Chantraine's formula \cite{ChantraineConcordance},
\[
\tb(\Lambda)=-\chi(L)=2g(L)-1.
\]
Since \(L\) is a smooth slice surface, \(g_s(\Lambda)\le g(L)\). On the other
hand, the slice--Bennequin inequality gives
\[
\tb(\Lambda)\le 2g_s(\Lambda)-1.
\]
See, for example, \cite{Et:LegenTranKnots}. Thus \(g(L)=g_s(\Lambda)\), and hence
\[
\tb(\Lambda)=2g_s(\Lambda)-1.
\]
Since \(L\) is not a disk, \(g(L)>0\), so \(\tb(\Lambda)\ne -1\). Therefore the first part of the corollary applies.
\end{proof}

\subsection{Adaptations for the ADC case and proof of \Cref{thm:vanishingCH1}}

Let $W$ be a symplectic filling such that
\begin{enumerate}
\item the natural map \(\pi_1(\partial W)\to \pi_1(W)\) is injective,
\item \(c_1(W)=0\), and
\item $\partial W$ is ADC.
\end{enumerate}
Then by \cite[Proposition 3.8]{Lazarev:flexiblefillings}, the positive symplectic cohomology \(SH^*_+(W)\) is independent of the filling \(W\). Analogous invariance results for ADC Legendrians are proved in \cite[Proposition 7.8]{Lazarev:flexiblefillings}. We will use the following augmented version of the latter result.

\begin{proposition}\label{prop:ind}
Let \((Y_+,\Lambda,\xi_+)\) be an ADC pair such that
\[
\pi_1(Y_+)=1,
\qquad
\pi_1(Y_+,\Lambda)=1.
\]
Let \(W\) be a Liouville cobordism from \(Y_-\) to \(Y_+\) with \(c_1(W)=0\). Let
\(L\subset W\) be an exact spin Lagrangian filling of \(\Lambda\) with vanishing
Maslov class. Finally, let \(\epsilon\) be a \(\Z\)-graded augmentation of the
contact homology algebra of \(Y_-\), with respect to a global trivialization of
\(\det_{\C}TW\). Then \(HW^*_+(L,\epsilon)\), as a \(\Z\)-graded vector space,
is independent of the choices of \(W\), \(L\), and \(\epsilon\).
\end{proposition}

\begin{proof}[Sketch of proof]
We may apply Lazarev's neck-stretching argument adapted to the augmented setting. Choose the sequence of contact forms and action thresholds from the ADC condition in \Cref{def:ADC_pair}, meaning that the SFT degrees of orbits and chords (up to action thresholds tending to infinity) are positive. Since \(W\) contains the negative symplectization end of \(Y_+\), these contact forms can be realized on contact hypersurfaces in \(W\), given as graphs in the collar end of the cobordism. 

Consider the moduli spaces \(\cM_{x,y,\Gamma}\) of Equation \eqref{eqn:wrapped_aug} that define the differential on \(HW^*_+(L,\epsilon)\). After stretching the neck along an ADC hypersurface, any curve entering the lower cobordism must break along closed Reeb orbits or Reeb chords of \((Y_+,\Lambda)\). The assumptions \(\pi_1(Y_+)=1\) and \(\pi_1(Y_+,\Lambda)=1\) ensure that the relevant orbits and chords lie in the homotopy classes to which the ADC positivity condition applies. Since the augmentation is \(\Z\)-graded, only degree-zero negative punctures at \(Y_-\) can contribute after weighting by \(\epsilon\). The ADC index positivity then rules out lower-level contributions to the rigid moduli spaces that define the differential and continuation maps.

Consequently, in the direct-limit complex computing \(HW^*_+(L,\epsilon)\), the differentials and continuation maps are computed entirely in the symplectization of \(Y_+\). Thus the resulting \(\Z\)-graded vector space depends only on the ADC pair \((Y_+,\Lambda,\xi_+)\), and not on the cobordism, the Lagrangian filling, or the augmentation.
\end{proof}

\begin{proof}[Proof of \Cref{thm:vanishingCH1}]
Since \(n\ge 2\), the surgery cobordism \(W_\Lambda\) is obtained by attaching an index $n+1$ handle to $\bbS^{2n+1}$ and therefore has \(c_1(W_\Lambda)=0\). Since $W_\Lambda$ and its boundary components all have vanishing first cohomology group, \(\det_{\C}\xi_{+1}(\Lambda)\) has a homotopically unique trivialization, yielding canonical $\Z$-valued gradings on chain-level contact homology algebras of the positive and negative boundaries of $W_\Lambda$.

We argue by contradiction, assuming that the chain-level contact homology algebra of
\[
(\bbS^{2n+1}_{+1}(\Lambda),\xi_{+1}(\Lambda))
\]
admits a \(\Z\)-graded augmentation \(\epsilon\) over \(\C\). Let \(L_0\subset W_\Lambda\) be the Lagrangian disk with positive boundary \(\Lambda\). The proof of \Cref{thm:vanishingCH} carries over without modification to show that the image of
\[
1\in H^0(L_0;\C)
\]
vanishes in \(HW^{*,<4\pi+\delta}(L_0,\epsilon)\). Hence the unit vanishes in the direct limit \(HW^*(L_0,\epsilon)\). Since \(HW^*(L_0,\epsilon)\) is a unital ring, it follows that
\[
HW^*(L_0,\epsilon)=0.
\]
The tautological exact sequence of Equation \eqref{eq:AugmentedWFLES} therefore yields
\[
HW^*_+(L_0,\epsilon)\cong H^*(L_0;\C)[1]\cong \C[1].
\]

On the other hand, consider the given exact Lagrangian filling
\[
L\subset(\D^{2n+2},\omega_{\std})
\]
of \(\Lambda\). Its wrapped Floer cohomology vanishes, since the symplectic cohomology of $\D^{2n+2}$  vanishes. Therefore the tautological exact sequence gives
\[
HW^*_+(L)\cong H^*(L;\C)[1].
\]
By assumption \(H^*(L;\C)\ne \C\), contradicting \Cref{prop:ind} applied to the ADC pair \((\bbS^{2n+1},\Lambda,\xi_{\std})\). Therefore no \(\Z\)-graded augmentation \(\epsilon\) can exist.
\end{proof}

\bibliographystyle{alpha}
\bibliography{ref}
\subsection*{AI declaration} LLMs were used for proofreading this work. All the rest was done by the authors.
\Addresses

\end{document}